\documentclass{amsart}
\usepackage{graphicx}
\usepackage{color}

\def\R{\mathbb R}
\def\N{\mathbb N}

\def\negquad{\!\!\!\!}
\def\negqquad{\!\!\!\!\!\!\!\!}
\def\eps{\varepsilon}

\newcommand{\cal}[1]{{\mathcal #1}}

\newtheorem{thm}{Theorem}

\newtheorem{lem}{Lemma}

\theoremstyle{definition}
\newtheorem{definition}{Definition}
\theoremstyle{remark}
\newtheorem{rem}{Remark}
\renewcommand{\proofname}{\textbf{Proof}}

\begin{document}
\title[Local Hadamard well--posedness...]
{Local Hadamard well--posedness and blow--up for reaction--diffusion
equations with non--linear dynamical boundary conditions}
\author{Alessio Fiscella}
\author{Enzo Vitillaro}
\address[E.~Vitillaro]
       {Dipartimento di Matematica ed Informatica, Universit\`a di Perugia\\
       Via Vanvitelli,1 06123 Perugia ITALY}
\email{enzo@dmi.unipg.it}
\address[A.~Fiscella]{Dipartimento di Matematica e Informatica, Universit\`a di Udine\\
Via delle Scienze, 206  33100 Udine ITALY}
\email{alessio.fiscella@uniud.it}

\date{\today}

\keywords{Heat equation, dynamical boundary conditions, reactive
terms. }


\begin{abstract}
The paper deals with local well--posedness, global existence and
blow--up results for reaction--diffusion equations coupled with
nonlinear dynamical boundary conditions. The typical problem studied
is $$\mbox{ $\left\{\begin{array}{llll}
$$u_{t}-\Delta u=\left|u\right|^{p-2}u & \mbox{in } \left(0,\infty\right)\times\Omega,$$\\
$$u=0 & \mbox{on } \left[0,\infty\right)\times\Gamma_{0},$$\\
$$\frac{\partial u}{\partial\nu} = -\left|u_{t}\right|^{m-2}u_{t} & \mbox{on } \left[0,\infty\right)\times\Gamma_{1},$$\\
$$u\left(0,x\right)=u_{0}\left(x\right) & \mbox{in } \Omega$$
\end{array}
\right.$} $$ where $\Omega$ is a bounded open regular domain of
$\mathbb{R}^{n}$ ($n\geq 1$), $\partial\Omega=\Gamma_0\cup\Gamma_1$,
$2\le p\le 1+2^*/2$, $m>1$ and $u_0\in H^1(\Omega)$,
${u_0}_{|\Gamma_0}=0$. After showing local well--posedness in the
Hadamard sense we give global existence and blow--up results when
$\Gamma_0$ has positive surface measure. Moreover we discuss the
generalization of the above mentioned results to more general
problems where the terms  $|u|^{p-2}u$ and $|u_{t}|^{m-2}u_{t}$ are
respectively replaced by $f\left(x,u\right)$ and $Q(t,x,u_t)$ under
suitable assumptions on them.
\end{abstract}

\maketitle

\section{Introduction and main results}
We consider the problem
\begin{equation}
\mbox{
$\left\{\begin{array}{llll}
$$u_{t}-\Delta u=f\left(x,u\right) & \mbox{in } \left(0,\infty\right)\times\Omega,$$\\

$$u=0 & \mbox{on } \left[0,\infty\right)\times\Gamma_{0},$$\\

$$\frac{\partial u}{\partial\nu} = -Q\left(t,x,u_{t}\right) & \mbox{on } \left[0,\infty\right)\times\Gamma_{1},$$\\

$$u\left(0,x\right)=u_{0}\left(x\right) & \mbox{in } \Omega$$
\end{array}
\right.$}\label{eq:P}
\end{equation}
where $u=u\left(t,x\right)$, $t\geq 0$, $x\in\Omega$,
$\Delta=\Delta_{x}$ denotes the Laplacian operator with respect to
the $x$ variable, $\Omega$ is a bounded open subset of
$\mathbb{R}^{n}$ ($n\geq 1$) of class $C^{1}$ (see \cite{brezis2}),
with $\partial\Omega=\Gamma_{0}\cup\Gamma_{1}$,
$\Gamma_{0}\cap\Gamma_{1}=\emptyset$, $\Gamma_{0}$ and $\Gamma_{1}$
are measurable over $\partial\Omega$, endowed with
$(n-1)$--dimensional surface measure $\sigma$. These properties of
$\Omega$, $\Gamma_{0}$ and $\Gamma_{1}$ will be assumed, without
further comments, throughout the paper. The initial datum $u_{0}$
belongs to the energy space $H^{1}(\Omega)$, with the compatibility
condition $u_{0}=0$ on $\Gamma_{0}$. Moreover $Q$ represents a
nonlinear dynamical term such that $Q(t,x,v)v\geq 0$, and $f$
represents a nonlinear internal reaction (or source) term, i.e.
$f(x,u) u\geq 0$.

When $Q\equiv 0$ problem (\ref{eq:P}) is an initial--boundary value
problem related to a  semilinear reaction--diffusion equation with
homogeneous Dirichlet -- Neumann boundary conditions. In this case
local well--posedness, under suitable assumptions on $f$, can be
obtained in a standard way using semigroup theory. See for example
\cite{lunardi, taylor3} or \cite{amann} combined with
\cite[Appendix]{chueshovellerlasiecka}. There is also a wide
literature on global existence and blow--up  for such type of
problems, starting from the classical paper of Levine
{\cite{levine3}. See for example
\cite{BCMR,galaktionovvazquez,ishigeyagisita,levinesiam,lps1,mizoguchi},
\cite[Section~5]{laplace} and \cite{dingguo,
jazarkiwan,payneschaefer2,payneschaefer1,paynesong}. In this case
the concavity method of H. Levine is effective in getting blow--up
results.

When $Q(t,x,u_{t})=\alpha(t,x)u_{t}$ problem (\ref{eq:P}) consists
in a reaction--diffusion equation coupled with a linear dynamical
boundary condition. For well--posedness results, obtained by
semigroup and interpolation theories we refer to  \cite{amann,
escher1,escher2,grobbelaar,hintermann}, while blow--up results were
proven in \cite{escher4,kirane}. We also refer to \cite{BDV} for a
physical motivation of dynamical boundary conditions, and to the
recent papers \cite{fanzhong,vonBelow1, vonBelow2}. Also in this
case the concavity method applies (see \cite{ps:private}) in order
to establish blow--up. We also would like to mention the classical
local--existence and blow--up results in
\cite{levpayne1,levpayne2,levsmith2,levsmith} dealing with the
related case when the source $f$ appears on the boundary condition.

When Q is nonlinear but monotone increasing in $u_{t}$ and, roughly
speaking, either $f\equiv 0$ or $-\Delta - f$ is a monotone
operator, the existence of global solutions of problem (\ref{eq:P})
can be proved by applying the results in \cite{colli}, since the
problem can be written as a doubly nonlinear evolution equation in a
suitable Banach space. We also refer to
\cite{gilardistefanelli,maitrewitomski,SSS} for related results.
When $Q$ is nonlinear and $f$ appears on the boundary condition
instead than in the equation, local and global existence has been
studied in \cite{globalheat}. Next, the same boundary condition
arises in the literature in connection with the wave equation, i.e.
when the heat operator $u_{t}-\Delta u$ in (\ref{eq:P}) is replaced
by the wave operator $u_{tt}-\Delta u$. In particular we refer to
\cite{autuoripucci,bociulasiecka2,bociulasiecka1,CDCL,CDCM,gerbi,lasiecka1,
stable}. Finally we would like to mention that our analysis on
global behavior of the solutions of \eqref{eq:P} is related to the
methods in \cite{lps1}. See also
\cite{APS,georgiev,levserr,ps:private}.

In this paper we study problem (\ref{eq:P}) when, roughly,
$Q(t,x,u_{t})\approx \left|u_{t}\right|^{m-2}u_{t}$ as
$\left|u_{t}\right|\geq 1$, $m >1$, and
$f(x,u)\approx\left|u\right|^{p-2}u$, $p\geq 2$, as
$\left|u\right|\geq 1$. The interest in considering superlinear
terms ($m> 2$) is mainly of theoretical nature. However, a physical
model involving $Q(t,x,u_{t})=u_{t}+\left|u_{t}\right|^{m-2}u_{t}$,
$m > 2$, is given in Appendix \ref{eq:appA}.

In order to state and prove our results in the simplest possible way
we shall first consider the model problem
\begin{equation}
\mbox{
$\left\{\begin{array}{llll}
$$u_{t}-\Delta u=\left|u\right|^{p-2}u & \mbox{in } \left(0,\infty\right)\times\Omega,$$\\
$$u=0 & \mbox{on } \left[0,\infty\right)\times\Gamma_{0},$$\\
$$\frac{\partial u}{\partial\nu} = -\left|u_{t}\right|^{m-2}u_{t} & \mbox{on } \left[0,\infty\right)\times\Gamma_{1},$$\\
$$u\left(0,x\right)=u_{0}\left(x\right) & \mbox{in } \Omega$$
\end{array}
\right.$}\label{Pmod}
\end{equation}
where $m>1$, $p\geq 2$. We  denote by $2^{*}$ the critical exponent
of Sobolev embedding $H^{1}(\Omega)\hookrightarrow L^{q}(\Omega)$,
i.e. $2^{*} = 2n/(n-2)$ when $n\geq 3$ while $2^{*} = \infty$ when
$n=1,2$. Moreover we denote
$\left\|\cdot\right\|_{q}=\left\|\cdot\right\|_{L^{q}(\Omega)}$,
$\left\|\cdot\right\|_{q,\Gamma_{1}}=\left\|\cdot\right\|_{L^{q}(\Gamma_{1})}$
for $1\leq q\leq\infty$, and the Hilbert space
$H^{1}_{\Gamma_{0}}\left(\Omega\right):=\left\{u\in
H^{1}\left(\Omega\right): u_{|\Gamma_{0}}=0\right\}$,
$\left\|u\right\|^{2}_{H^{1}_{\Gamma_{0}}}:=\left\|u\right\|^{2}_{2}+\left\|\nabla
u\right\|^{2}_{2}$,  where $u_{|\Gamma_0}$ stands for the
restriction of the trace of $u$ on $\partial\Omega$ to $\Gamma_0$.
The first aim of the paper is to show that problem (\ref{Pmod}) is
well--posed in
 $H^{1}_{\Gamma_{0}}(\Omega)$.
 The first step in this direction is given by the following  result.

\begin{thm} \label{Thloc} {\bf (Local existence and uniqueness)}
Let $m>1$ and
\begin{equation}
2\leq p\leq 1+\frac{2^{*}}{2}.\label{eq:pr}
\end{equation}
Then, given $u_{0}\in H^{1}_{\Gamma_{0}}(\Omega)$, there is a
$T^{*}=T^{*}(\left\|u_{0}\right\|_{H^{1}_{\Gamma_{0}}},m,p,\Omega,\Gamma_{1})
\in(0,1]$, decreasing in the first variable, such that problem
(\ref{Pmod}) has a unique weak
\begin{footnote}{see Definition~\ref{eq:DEF} below for the precise meaning of weak solutions, which are essentially distributional solutions
enjoying  a suitable regularity}\end{footnote} solution $u$ in
\mbox{$[0,T^{*}]\times\Omega$}. Moreover
\begin{equation}
u\in C(\left[0,T^{*}\right];H^{1}_{\Gamma_{0}}(\Omega)), \label{eq:4.12}
\end{equation}
\begin{equation}
u_{t}\in L^{m}((0,T^{*})\times\Gamma_{1})\cap L^{2}((0,T^{*})\times\Omega) \label{eq:4.13}
\end{equation}
and the energy identity
\begin{equation}
\frac{1}{2}\left\|\nabla u\right\|^{2}_{2}\Big|
^{t}_{s}+\int^{t}_{s}\left\|u_{t}\right\|^{m}_{m,\Gamma_{1}}+\left\|u_{t}\right\|^{2}_{2}=\int^{t}_{s}\int_{\Omega}\left|u\right|^{p-2}uu_{t}
\label{eq:ei2}
\end{equation}
holds for $0\leq s\leq t\leq T^{*}$.
Finally
\begin{equation}
\left\|u\right\|_{C([0,T^{*}];H^{1}_{\Gamma_{0}}(\Omega))}\leq 4\left\|u_{0}\right\|_{H^{1}_{\Gamma_{0}}}. \label{eq:gengis}
\end{equation}

\end{thm}

\begin{rem}
The assumption $p\leq 1+2^{*}/2$ in Theorem \ref{Thloc} is quite
restrictive when $n\geq 3$, although it appears often in the
literature quoted above. Clearly it expresses the assumption that
the Nemitski operator $u\mapsto\left|u\right|^{p-2}u$ is locally
Lipschitz from $H^{1}(\Omega)$ to $L^{2}(\Omega)$. Such type of
assumptions has been overcome, in the author's knowledge, either by
getting additional a--priori estimates, as done for example in
\cite{bociulasiecka2,STV}, or using linear semigroup and
interpolation theories, as done for example in \cite{amann,escher2}.
While in this case the nonlinear term $Q$ does not give useful
estimates, being active on the boundary, it prevents to use linear
theory and interpolation of semigroups. Nonlinear semigroup theory
can be used, as in \cite{colli}, but in this case one still needs to
assume that the Nemistski operator above is  locally Lipschitz, as
in \cite{chueshovellerlasiecka}. To prove Theorem~\ref{Thloc} we
found simpler to  first use the monotonicity method of J. L. Lions
and then to use a contraction argument.
\end{rem}

By using the same energy estimates used to prove Theorem \ref{Thloc}
we complete our well--posedness analysis as follows.

\begin{thm} \label{Thalt} {\bf(Continuation and local Hadamard  well--posedness)}
Under the assumption of Theorem \ref{Thloc}, problem (\ref{Pmod})
has a unique weak maximal solution $u$ in $[0,T_{max})\times\Omega$.
Moreover $u\in
C(\left[0,T_{max}\right);H^{1}_{\Gamma_{0}}(\Omega))$,
$$u_{t}\in
L^{m}((0,T)\times\Gamma_{1})\cap L^{2}((0,T)\times\Omega)\quad\text{
for any $T\in(0, T_{max})$,}$$ and the following alternative holds:
\renewcommand{\labelenumi}{(\roman{enumi})}
\begin{enumerate}
    \item  either $T_{max}=\infty$;
    \item  or $T_{max} < \infty$ and
\begin{equation}\label{bu}
\lim_{t \rightarrow T^{-}_{max}}
\left\|u(t)\right\|_{H^{1}_{\Gamma_{0}}} = +\infty.
\end{equation}
\end{enumerate}
Finally $u$ depends continuously on the initial datum $u_{0}$, that
is given any $T\in(0, T_{max})$ and any sequence $(u_{0n})_{n}$ in
$H^{1}_{\Gamma_{0}}(\Omega)$ such that $u_{0n}\rightarrow u_{0}$ in
$H^{1}_{\Gamma_{0}}(\Omega)$, the corresponding weak solution $u^n$
is defined in $\left[0,T\right]\times\Omega$ and
    $
u^n\rightarrow u \mbox{ in }
C(\left[0,T\right];H^{1}_{\Gamma_{0}}(\Omega))$.
\end{thm}
The second aim of the paper is to study the alternative (i)--(ii) in
previous Theorem by giving global existence versus blow--up results.
When $p=2$ it is straightforward to prove that $u$ is global (see
Theorem~\ref{easytheorem} in Appendix~\ref{appendicep2}), so we
focus on the more interesting case $p>2$. Although we are not able
to give a complete answer, as usual for nonlinear problems, we give
two partial answers when
\begin{equation}
\sigma(\Gamma_{0}) > 0, \label{eq:sig}
\end{equation}
so a Poincar\`e--type inequality holds (see \cite{ziemer}) and
consequently  $\left\|\nabla u\right\|_{2}$ is an equivalent norm in
$H^{1}_{\Gamma_{0}}(\Omega)$. This assumption allows us to use
potential--well arguments. In order to state our next results we
need to recall the stable and unstable sets introduced in
\cite{stable}. When $p>2$ and \eqref{eq:pr} holds  we introduce the
functionals
\begin{equation}
J(u)=\frac{1}{2}\left\|\nabla
u\right\|^{2}_{2}-\frac{1}{p}\left\|u\right\|^{p}_{p},\qquad
K(u)=\left\|\nabla u\right\|^{2}_{2}-\left\|u\right\|^{p}_{p}
\label{eq:K}
\end{equation}
defined for $u\in H^{1}_{\Gamma_{0}}(\Omega)$, and the number
\begin{equation}
d=\inf_{u\in
H^{1}_{\Gamma_{0}}(\Omega)\setminus\{0\}}\,\sup_{\lambda
>0}J(\lambda u). \label{eq:d}
\end{equation}
When $p>2$ and \eqref{eq:pr}, \eqref{eq:sig} hold true it is easy to
see that $d>0$. See Lemma \ref{eq:Ld} below, where two different
characterizations of $d$ are given. We define the stable and
unstable sets as
\begin{equation}W_{s}=\left\{u_{0}\in H^{1}_{\Gamma_{0}}(\Omega): K(u_{0})\geq 0 \mbox{ and } J(u_{0})<d\right\}
\label{eq:Ws}
\end{equation}
\begin{equation}W_{u}=\left\{u_{0}\in H^{1}_{\Gamma_{0}}(\Omega): K(u_{0})\leq 0 \mbox{ and } J(u_{0})<d\right\}. \label{eq:Wu}
\end{equation}
As an application of Theorem \ref{Thalt} and of a potential--well
estimate we give the following global existence result.
\begin{thm}\label{Thglo}{\bf (Global existence)}
Under the assumptions of Theorem \ref{Thloc} and the further
assumptions \eqref{eq:sig} and $p>2$, if $u_{0}\in W_{s}$ then
$T_{max}=\infty$ and $u(t)\in W_{s}$ for all $t\geq 0$.
\end{thm}
While Theorem \ref{Thglo} can be seen as a simple application of
Theorem \ref{Thalt}, to recognize that solutions of problem
(\ref{Pmod}) starting in the unstable set blow--up is a more
difficult task. When $m=2$ this  result can be proved by a concavity
argument (see \cite{rendiconti}), which cannot be applied when
$m\neq 2$, making this case more interesting. By combining the main
technique of \cite{lps1} with an estimate used in  \cite{stable} for
wave equation we are able to prove the following result.
\begin{figure}
\includegraphics[width=12truecm]{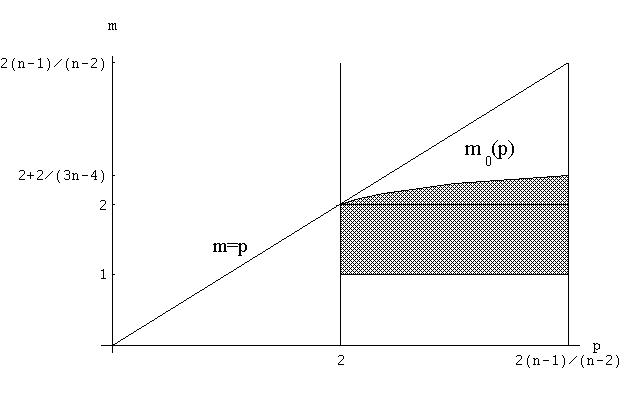}
 \caption{The shaded region
is the set of the $(p,m)$ couples for which the assumptions of
Theorem~\ref{Thblow} hold, when $n\ge 3$. The picture is made in the
case $n=3$.} \label{fig1}
\end{figure}
\begin{thm} \label{Thblow}{\bf (Blow--up)}
Under the assumptions of Theorem \ref{Thloc} and the further
assumptions \eqref{eq:sig},  $p>2$  and
\begin{equation}
m < m_{0}(p):=\frac{2(n+1)p-4(n-1)}{n(p-2)+4}, \label{eq:w4}
\end{equation}
if $u_{0}\in W_{u}$ then  $T_{\max} < \infty$,  $u(t)\in W_{u}$ for
all $t\in \left[0,T_{max}\right)$, and
    $\lim_{t\rightarrow T^{-}_{\max}}\left\|u(t)\right\|_p=
    +\infty$.
\end{thm}
\begin{rem}
Even if is not evident from \eqref{eq:Ws} and \eqref{eq:Wu},
$W_{s}\cap W_{u}=\emptyset$ (see Lemma \ref{eq:Ld} below),  so
Theorems \ref{Thglo} and \ref{Thblow} are consistent.
\end{rem}
\begin{rem}\label{servo}
Clearly assumption (\ref{eq:w4}) yields $m<p$ since it is trivial to
prove that $m_0(p)\le p$ for $p\ge 2$. It strongly reduces the
applicability of Theorem \ref{Thblow}, as shown by Figure~\ref{fig1}
which illustrates the set of the couples $(p,m)$ satisfying
(\ref{eq:pr}) and (\ref{eq:w4}). As $m_{0}(p)>2$ for $p>2$, the
result is rather sharp in the sublinear case $1<m\leq 2$, while
(\ref{eq:pr}) and (\ref{eq:w4}) force that $m<4$ when $n=1$, $m<3$
when $n=2$ and $m<2+\frac{2}{3n-4}$ when $n\geq 3$. This assumption,
which looks to be a technical one, comes directly from
\cite{stable}, where it was introduced, and is due to the difficulty
in comparing the effect of high order polynomial dissipation, which
is related to the $L^{m}$ norm on $\Gamma_{1}$, with the effect of
the source, related to the $L^{p}$ norm on $\Omega$. After nine
years from its use, the authors are not aware of any improvement.
\end{rem}

As a preliminary step in the proof of Theorem~\ref{Thloc} we give a
well--posedness result for the problem
\begin{equation}
\mbox{ $\left\{\begin{array}{llll}
$$u_{t}-\Delta u=g(t,x) & \mbox{in } \left(0,T\right)\times\Omega,$$\\
$$u=0 & \mbox{on } \left[0,T\right)\times\Gamma_{0},$$\\
$$\frac{\partial u}{\partial\nu} = -\left|u_{t}\right|^{m-2}u_{t} & \mbox{on } \left[0,T\right)\times\Gamma_{1},$$\\
$$u(0,x)=u_{0}(x) & \mbox{in } \Omega$$
\end{array}
\right.$}\label{eq:P1}
\end{equation}
{where $m > 1$, $T > 0$ is arbitrary and $g$ is a given forcing term
acting on $\Omega$. Although problem \eqref{eq:P1} can be studied
using the analysis of \cite{colli}, it is not trivial in that way to
get the following result.

\begin{thm} \label{eq:T1} {\bf (Well--posedness for an auxiliary problem)}
Suppose that $u_{0}\in H^{1}_{\Gamma_{0}}(\Omega)$ and $g\in
L^{2}((0,T)\times\Omega)$. Then there is a unique weak
\begin{footnote}{see Definition~\ref{eq:def1} below for the precise meaning of weak solution}\end{footnote} solution $u$ of
$(\ref{eq:P1})$ in $[0,T]\times\Omega$. Moreover
\begin{equation}
u\in C(\left[0,T\right];H^{1}_{\Gamma_{0}}(\Omega)), \label{eq:4.1}
\end{equation}
\begin{equation}
u_{t}\in L^{2}((0,T)\times\Omega)\cap L^{m}((0,T)\times\Gamma_{1})
\label{eq:4.2}
\end{equation}
and the energy identity
\begin{equation}
\frac{1}{2}\left\|\nabla u\right\|^{2}_{2}\Big
|^{t}_{s}+\int^{t}_{s}\left\|u_{t}\right\|^{2}_{2}+\left\|u_{t}\right\|^{m}_{m,\Gamma_{1}}=\int^{t}_{s}\int_{\Omega}gu_{t}
\label{eq:ei}
\end{equation}
holds for $0\leq s\leq t\leq T$. Finally, given any couple of
initial data $u_{01}, u_{02}\in H^{1}_{\Gamma_{0}}(\Omega)$ and any
couple of forcing terms $g_1,g_2\in L^{2}((0,T)\times\Omega)$,
respectively denoting by $u^1$ and $u^2$ the solutions of
\eqref{eq:P1}
  corresponding to $u_{01}$,  $g_1$ and to $u_{02}$,
$g_2$, the following estimate holds
\begin{equation}\label{hadamardaux}
\|u^1-u^2\|^2_{C(\left[0,T\right];H^{1}_{\Gamma_{0}}(\Omega))}\le
2(1+T)\left(\|u_{01}-u_{02}\|_{
H^{1}_{\Gamma_{0}}}^2+\|g_1-g_2\|_{L^2((0,T)\times\Omega)}^2\right).
\end{equation}
\end{thm}

\begin{rem} A short comparison with the results which can be obtained by directly applying the abstract
results in \cite{colli} is in order. Assumptions (A1--2) in
\cite[Theorem~1]{colli} force to restrict to the case $m=2$, while
the assumption $D(B)\subset V$ in \cite[Theorems~2--3]{colli}
implies $m\le 2(n-1)/(n-2)$ when $n\ge 3$. Next one can apply
\cite[Theorem~4]{colli} only when $g$ is more regular in time.
Finally, \cite[Theorem~5]{colli} can be applied only when $m=2$.
\end{rem}

In order to explain the main difficulties arising in the proofs of
our main results we now make some comparison with the arguments used
by the second author in \cite{globalheat}. Theorem~\ref{eq:T1} is
essentially proved as \cite[Theorem~1.5]{globalheat}, even if the
necessary adaptations require some care. Theorem~\ref{Thloc} is
proved by a contraction argument instead that a compactness one.
Theorem~\ref{Thalt} has no counterpart in \cite{globalheat}. Finally
the proof of Theorem~\ref{Thblow} requires an untrivial mixing of
the technique of \cite{lps1} with the estimate used in
\cite{stable}, so the authors consider it as the main contribution
in the present paper.

The paper is organized as follows. Section~\ref{section2} deals with
some notation and preliminary material, including the proof of
Theorem~\ref{eq:T1}, Section~\ref{section3} is devoted to local
well--posedness theory for problem \eqref{Pmod} while in
Section~\ref{section4} we study global existence and  blow--up for
it. Finally the results presented in this introduction are
generalized in Section~\ref{generalizations} to problem
\eqref{eq:P}, under suitable assumptions on the nonlinearities $f$
and $Q$. For the sake of simplicity we first present the proofs for
the model problem \eqref{Pmod} and then we give in
\eqref{generalizations} the generalizations needed to handle with
\eqref{eq:P}. This section is naturally addressed to a more
specialized audience and consequently an higher lever of
mathematical expertise of the reader is supposed. In particular most
proofs are only sketched.

\section{Notation and preliminaries}\label{section2}
We introduce the notations

\begin{tabular}{ll}
$C_{c}^\infty (\cal{O})$ & space of compactly supported real--valued $C^{\infty}$ functions \\
&on any open set  $\cal{O}\subset\R^n$,\\
$C^\infty ((a,b);X)$ & space of $C^{\infty}$ $X$--valued functions in  $(a,b)$, $X$ Banach space, \\
$C(\left[a,b\right];X)$ & space of  norm continuous $X$--valued functions in  $\left[a,b\right]$,\\
$C_{w}(\left[a,b\right];X)$ & space of weakly continuous $X$--valued functions in  $\left[a,b\right]$,\\
$q'$ & H\"{o}lder conjugate of $q \ge 1$, i.e. $1/q+1/q'=1$,\\
$X^{'}$ & the dual space of $X$,\\
$(\cdot,\cdot)$ & scalar product in
$L^{2}(\Omega)$.
\end{tabular}

\noindent Moreover we call \textit{the trace theorem} the existence
of the continuous trace mapping
$H^{1}_{\Gamma_{0}}(\Omega)\hookrightarrow L^{2}(\partial\Omega)$.
Moreover the trace of $u$ on $\Omega$ will be denoted by
$u_{|\partial\Omega}$.
 We also call \textit{the Sobolev Embedding Theorem}
the existence of the continuous embedding
    $H^{1}_{\Gamma_{0}}(\Omega)\hookrightarrow L^{p}(\Omega) \mbox{ for } 2\leq p <2^*$.

\noindent We start by setting the Banach space
\begin{equation}\label{Xdef}
X=\left\{u\in H^{1}_{\Gamma_{0}}(\Omega):u_{|\Gamma_{1}}\in
L^{m}(\Gamma_{1})\right\}
\end{equation}
endowed with the norm
$\|u\|_X=\|u\|_{H^1_{\Gamma_0}}+\|u_{|\Gamma_1}\|_{m,\Gamma_1}$. For
elements $u\in X$ we shall use the simpler notation
$\|u\|_{m,\Gamma_1}$ to mean  $\|u_{|\Gamma_1}\|_{m,\Gamma_1}$.
 We now give the precise precise meaning of weak
solution of \eqref{eq:P1}.
\begin{definition}\label{eq:def1} Let $u_{0}\in H^{1}_{\Gamma_{0}}(\Omega)$ and $g\in L^{2}((0,T)\times\Omega)$.
We say that $u$  is a weak solution of $(\ref{eq:P1})$ in
$[0,T]\times\Omega$ if
\renewcommand{\labelenumi}{(\alph{enumi})}
\begin{enumerate}
\item   $u$ $\in L^{\infty}(0,T;H^{1}_{\Gamma_{0}}(\Omega)),u_{t}\in L^{2}((0,T)\times\Omega);$

\item the spatial trace of $u$ on
$(0,T)\times\partial\Omega$ (which exists by the trace theorem) has
a distributional time derivative on $(0,T)\times\partial\Omega$,
belonging to \mbox{$L^m((0,T)\times\partial\Omega)$};
\item  for all $\phi\in X$ and for almost all \noindent$t\in\left[0,T\right]$ the
distribution identity
\begin{equation}
\int_{\Omega}u_{t}(t)\phi + \int_{\Omega}\nabla u(t)\nabla \phi +
\int_{\Gamma_{1}}\left|u_{t}(t)\right|^{m-2}u_{t}(t)\phi=\int_{\Omega}g(t)\phi
\label{eq:def2}
\end{equation}
holds true;
\item $u(0)=u_{0}$.
\end{enumerate}

\end{definition}
} Note that, in (d), $u(0)$ makes sense since, by (a),
$$
u\in H^{1}(0,T;L^{2}(\Omega))\hookrightarrow C(\left[0,T\right];L^{2}(\Omega)).
$$
In order to prove Theorem~\ref{eq:T1} we need the following Lemma,
which extends \cite[Theorems~3.1 and 3.2]{strauss} to the present
situation. Its proof consists in a rather technical application of
the arguments in \cite{strauss} which is given in
Appendix~\ref{Appendice A} for the reader convenience.

\begin{lem} \label{eq:l1} Let $0<T<\infty$, $m>1$,
\begin{equation}
u_0\in H^1_{\Gamma_0}(\Omega), \quad g\in L^2((0,T)\times\Omega),
\quad \zeta\in L^{m'}((0,T)\times\Gamma_{1}) \label{eq:b8}
\end{equation}
and suppose that $u$ is a weak solution of
\begin{equation}
\mbox{
$\left\{\begin{array}{llll}
$$u_{t}-\Delta u=g(t,x) & \mbox{in } (0,T)\times\Omega,$$\\
$$u=0 & \mbox{on } \left[0,T\right)\times\Gamma_{0},$$\\
$$\frac{\partial u}{\partial\nu} = \zeta & \mbox{on } \left[0,T\right)\times\Gamma_{1},$$\\
$$u(0,x)=u_{0}(x) & \mbox{in } \Omega$$
\end{array}
\right.$}\label{eq:P2}
\end{equation}
i.e. a function
\begin{equation}
u\in L^{\infty}(0,T;H^{1}_{\Gamma_{0}}(\Omega)) \label{eq:b1}
\end{equation}
such that
\begin{equation}
u_{t}\in L^{2}((0,T)\times\Omega), \label{eq:b2}
\end{equation}
the spatial trace of $u$ on $(0,T)\times\partial\Omega$
$(\mbox{which exists by the trace theorem}) $ has a distributional
time derivate on $(0,T)\times\partial\Omega$ belonging to
$L^{m}((0,T)\times\partial\Omega)$, and, for all $\phi\in X$ and
almost all $t\in\left[0,T\right]$ the function $u$ satisfies
\begin{equation}
\int_{\Omega}u_{t}(t)\phi + \int_{\Omega}\nabla u(t)\nabla \phi - \int_{\Gamma_{1}}\zeta(t)\phi=\int_{\Omega}g(t)\phi. \label{eq:b4}
\end{equation}
Then
\begin{equation}
u\in C(\left[0,T\right];H^{1}_{\Gamma_{0}}(\Omega)) \label{eq:1.9}
\end{equation}
and the energy identity
\begin{equation}
\frac{1}{2}\left\|\nabla
u\right\|^{2}_{2}\Big|^{t}_{s}+\int^{t}_{s}\left\|u_{t}\right\|^{2}_{2}-\int^{t}_{s}\int_{\Gamma_{1}}\zeta
u_{t}=\int^{t}_{s}\int_{\Omega}gu_{t} \label{eq:1.10}
\end{equation}
holds for $0\leq s\leq t\leq T$.
\end{lem}

\begin{proof}[\bf Proof of Theorem \ref{eq:T1}] To prove the existence of a weak solution of \eqref{eq:P1} we apply the Faedo--Galerkin procedure.
Let $(w_{k})_{k}$ be a sequence of linearly independent vectors in
the space  $X$, which was defined in \eqref{Xdef}, whose finite
linear combinations are dense in it. By using the Graham--Schmidt
orthonormalization process, we can take $(w_{k})_{k}$ to be
orthonormal in $L^{2}(\Omega)$. Since (see \cite[Lemma A1, Appendix
A]{global})  $X$ is dense in $H^{1}_{\Gamma_{0}}(\Omega)$ for all
$k\in\mathbb{N}$ there are real numbers $y^{j}_{0k}$,
$j=1,\ldots,k$, such that
\begin{equation}
u_{0k}=\sum\limits^{k}_{j=1}y^{j}_{0k}w_{j}\rightarrow u_{0}\mbox{
in }H^{1}_{\Gamma_{0}}(\Omega). \label{eq:1.4}
\end{equation}

\noindent For any fixed $k\in\mathbb{N}$ we look for approximate
solutions of $(\ref{eq:P1})$, that is for solutions
$u^k(t)=\sum\limits^{k}_{j=1}y^{j}_k(t) w_{j}$, of the
finite--dimensional problem
\begin{equation}
\mbox{
$\left\{\begin{array}{ll}
$$(u^{k}_{t},w_{j})+(\nabla u^{k}_{t},\nabla w_{j})+\int_{\Gamma_{1}}\left|u^{k}_{t}\right|^{m-2}u^{k}_{t}w_{j}=\int_{\Omega}gw_{j}, & \mbox{    } j=1,\ldots,k,$$\\

$$u^{k}(0)=u_{0k}. $$
\end{array}
\right.$}\label{eq:1.1}
\end{equation}
In order to recognize that \eqref{eq:1.1} has a local solution, we
set
\begin{eqnarray}
y_{0k} &=& (y^{1}_{0k},\ldots,y^{k}_{0k})^{T},\qquad\qquad\,
y_{k} = (y^{1}_{k},\ldots,y^{k}_{k})^{T},\\
A_{k} &=& ((\nabla w_{i},\nabla w_{j}))_{i,j=1,\ldots,k},\quad
B_{k}(x) = (w_{1}(x),\ldots,w_{k}(x))^{T},\\
G_{k}(y) &=& y+\int_{\Gamma_{1}}\left|B_{k}(x)\cdot
y\right|^{m-2}B_{k}(x)\cdot y B_{k}(x)dx, \mbox{     }
y\in\mathbb{R}^{k},\label{sdai}
\end{eqnarray}
and $ H_{k}(t)=\int_{\Omega}g(t,x)B_{k}(x)dx$,  so problem
\eqref{eq:1.1} can be rewritten as
\begin{equation}
\mbox{
$\left\{\begin{array}{l}
$$G_{k}(y'_{k}(t))+A_{k}y_{k}(t)=H_{k}(t), $$\\
$$y_{k}(0)=y_{0k}. $$
\end{array}
\right.$}\label{eq:1.2}
\end{equation}
Then, using the arguments in \cite[Proof of Theorem 1.5]{globalheat}
we get that $G_k$ is an homeomorphism from $\R^k$ into iteself, with
inverse $G_k^{-1}$, and that  \eqref{eq:1.2}  has a  solution
$y_k\in W^{1,1}(0,t_{k})$ for some $t_{k}\in (0,T]$, and
consequently \eqref{eq:1.1} has a solution $u^k\in
W^{1,1}(0,t_k;X)$. Moreover, since $G_{k}(y)y\geq
\left|y\right|^{2}$ for all $y\in\mathbb{R}^{k}$, by the Schwartz
inequality it follows that $\left|y\right|\leq
\left|G_{k}(y)\right|$. Then $\left|G^{-1}_{k}(y)\right|\leq
\left|y\right|$ for all $y\in\mathbb{R}^{k}$, so that
\begin{equation}
\left|G^{-1}_{k}(H_{k}(t)-A_{k}y_{k}(t))\right|\leq \left|H_{k}(t)\right|+ \left\|A_{k}\right\|\left|y_{k}\right|. \label{eq:1.7}
\end{equation}

Multiplying $(\ref{eq:1.1})$ by $(y^{j}_{k})'$ and summing for $j=1,\ldots,k$, we obtain the energy identity (here and in the sequel, explicit dependence on $t$ will be omitted, when clear)
\begin{equation}
\frac{d}{dt}\left(\frac{1}{2}\left\|\nabla u^{k}\right\|^{2}_{2}\right)+\left\|u^{k}_{t}\right\|^{2}_{2}+
\left\|u^{k}_{t}\right\|^{m}_{m,\Gamma_{1}}=\int_{\Omega}gu^{k}_{t}. \label{eq:1.12}
\end{equation}
Integrating over $(0,t)$, $0 < t < t_{k}$, and using Young inequality, we get
\begin{equation*}
\frac{1}{2} \left\| \nabla u^{k}\right\|^{2}_{2} +\int^{t}_{0}\left(\left\|u^{k}_{t}\right\|^{2}_{2}+ \left\| u^{k}_{t}\right\|^{m}_{m,\Gamma_{1}}\right)\leq \frac{1}{2} \left\|\nabla u_{0k}\right\|^{2}_{2} + \frac{1}{2} \left\|g\right\|^{2}_{L^{2}((0,T)\times\Omega)} +\frac{1}{2}\int^{t}_{0} \left\| u^{k}_{t}\right\|^{2}_{2}.
\end{equation*}
Then, using $(\ref{eq:1.4})$, there exists
    $C=C\left(\left\|\nabla
    u_{0}\right\|_{2},\left\|g\right\|_{L^{2}((0,T)\times\Omega)}\right)>0$
such that
\begin{equation}
\begin{alignedat}2
&\|\nabla u^{k}\|_{L^{\infty}(0,t_{k};L^{2}(\Omega))} && \leq C, \\
&\|u^{k}_{t}\|_{L^{2}((0,t_{k})\times\Omega)} && \leq C, \\
&\|u^{k}_{t}\|_{L^{m}((0,t_{k})\times\Gamma_{1})} && \leq C, \\
&\||u^{k}_{t}|^{m-2}u^{k}_{t}\|_{L^{m'}((0,t_{k})\times\Gamma_{1})}&&
\leq C,
\end{alignedat}
\label{eq:1.5}
\end{equation}
for $k\in\mathbb{N}$. By  $(\ref{eq:1.4})$,  $(\ref{eq:1.5})$ and
H\"{o}lder inequality in time it follows that
\begin{equation}
\left\|u^{k}\right\|_{2}\leq
\left\|u_{0k}\right\|_{2}+\int^{t}_{0}\left\|u^{k}_{t}\right\|_{2}
\leq \left\|u_{0k}\right\|_{2}+
T^{1/2}\left(\int^{t}_{0}\left\|u^{k}_{t}\right\|_{2}\right
)^{1/2}\leq C' \label{eq:1.6}
\end{equation}
for some $
C'=C'\left(\left\|u_{0}\right\|_{H^{1}_{\Gamma_{0}}},\left\|g\right\|_{L^{2}((0,T)\times\Omega)},T\right)>0$.
 Since $(w_{k})_{k}$ is orthonormal in $L^{2}(\Omega)$, we have
 $\left|y_{k}(t)\right|=\left\|u^{k}(t)\right\|_{2}$,
 so $(\ref{eq:1.6})$ yields that $\left|y_{k}(t)\right|\leq C'$. Then, by \eqref{eq:1.7}
\[\left|G^{-1}_{k}(H_{k}(t)-A_{k}y_{k}(t))\right|\leq \left|H_{k}(t)\right|+
C' \left\|A_{k}\right\| \in L^{1}(0,T).
\]
We can then apply \cite[Theorem 1.3,~Chapter 2]{coddington} to
conclude that $t_{k}=T$ for $k=1,\ldots,n$. Next, by
$(\ref{eq:1.5})$ and $(\ref{eq:1.6})$, it follows that, up to a
subsequence,
\begin{equation}
\begin{alignedat}2
u^{k}&\rightarrow u && \mbox{  weakly* in  } L^{\infty}(0,T;H^{1}_{\Gamma_{0}}(\Omega)), \\
u^{k}_{t}&\rightarrow u_{t} && \mbox{  weakly in  } L^{2}((0,T)\times\Omega), \\
u^{k}_{t}&\rightarrow \varphi && \mbox{  weakly in  } L^{m}((0,T)\times\Gamma_{1}), \\
\left|u^{k}_{t}\right|^{m-2}u^{k}_{t}&\rightarrow \chi \quad &&
\mbox{ weakly in  } L^{m'}((0,T)\times\Gamma_{1}).
\end{alignedat}
\label{eq:1.8}
\end{equation}
A consequence of the convergences \eqref{eq:1.8} and of Aubin--Lions
compactness Lemma (see \cite{breziscazenave, aubin, simon}) is that
$u^{k}\rightarrow u$ strongly in
$C(\left[0,T\right];L^{2}(\Omega))$, so that $u(0)=u_{0}$. It
follows in a standard way (see, for example, \cite[p. 272]{global})
that $\varphi$ is the distribution time derivate of $u$ on
$(0,T)\times\partial\Omega$, i.e. $\varphi=u_{t}$.

Next, multiplying $(\ref{eq:1.1})$ by $\phi\in C^{\infty}_{c}(0,T)$,
integrating on $(0,T)$, passing to the limit as $k\rightarrow\infty$
(using $\eqref{eq:1.8}$) and finally using the density of the finite
linear combinations of $(w_{k})_{k}$ in $X$, we obtain $
\int^{T}_{0}\left[(u_{t},w)+(\nabla u_{t},\nabla
w)+\int_{\Gamma_{1}}\chi w-\int_{\Omega}gw\right]\phi=0$ for all
$w\in X$, $\phi\in C^{\infty}_{c}(0,T)$. Consequently
    $(u_{t},w)+(\nabla u,\nabla w)+\int_{\Gamma_{1}}\chi w=\int_{\Omega}gw
$ almost everywhere in $(0,T)$. Then to prove that $u$ is a weak
solution of $(\ref{eq:P1})$ we have only to show that
\begin{equation}
\chi=\left|u_{t}\right|^{m-2}u_{t}\mbox{  a.e. on }(0,T)\times\Gamma_{1}. \label{eq:1.11}
\end{equation}
By Lemma \ref{eq:l1} we obtain $(\ref{eq:1.9})$ and the energy
identity
\begin{equation}
\frac{1}{2}\left\|\nabla u\right\|^{2}_{2}\Big
|^{T}_{0}+\int^{T}_{0}\left\|u_{t}\right\|^{2}_{2}+\int^{T}_{0}\int_{\Gamma_{1}}\chi
u_{t}=\int^{T}_{0}\int_{\Omega}gu_{t} \label{eq:1.13}
\end{equation}

The classical monotonicity method (see~\cite{lions} or  \cite[p.
186]{globalheat}) then allows us to prove $(\ref{eq:1.11})$.

Finally, to prove the estimate \eqref{hadamardaux}, which also
yields the uniqueness of the solution, we recognize that $v=u_1-u_2$
is a weak solution of problem \eqref{eq:P2} with $g=g_1-g_2$, $\xi=
-|u^1_t|^{m-2}u^1_t+|u^2_t|^{m-2}u^2_t$ and $u_0=u_{01}-u_ {02}$.
Then, by Lemma \ref{eq:l1}, using the monotonicity of the map $x
\rightarrow \left|x\right|^{m-2}x$  we get the estimate
$$\frac 12 \|\nabla v(t)\|_2^2+\int_0^t \|v_t\|_2^2\le \int_0^t g v_t+\frac
12 \|\nabla u_0\|_2^2\qquad \text{for all $t\in [0,T)$}.$$ By Young
inequality
$$  \|\nabla v(t)\|_2^2+ \int_0^t \|v_t\|_2^2\le
\|g\|_{L^2((0,T)\times\Omega)}^2+ \|\nabla u_0\|_2^2\qquad \text{for
all $t\in [0,T)$}.
$$
Moreover $\|v(t)\|_2^2\le
\left(\|u_0\|_2+\int_0^t\|v_t\|_2\right)^2\le
2\|u_0\|_2^2+2T\int_0^T \|v_t\|_2^2$. By combining the last two
estimates we get \eqref{hadamardaux} and conclude the proof.
\end{proof}

\section{Proofs of Theorems~\ref{Thloc} and
\ref{Thalt}.}\label{section3}

This section is devoted to prove our main well--posedness
Theorems~\ref{Thloc} and \ref{Thalt}. We first precise the meaning
of weak solution for problem \eqref{Pmod}.
\begin{definition} \label{eq:DEF} Let $u_{0}\in
H^{1}_{\Gamma_{0}}(\Omega)$. When assumption \eqref{eq:pr} holds we
say that $u$ is a weak solution of problem \eqref{Pmod} in
$[0,T]\times\Omega$ if (a--d) of Definition~\ref{eq:def1} hold, with
the distribution identity \eqref{eq:def2} being replaced by
\begin{equation}
\int_{\Omega}u_{t}(t)\phi + \int_{\Omega}\nabla u(t)\nabla \phi + \int_{\Gamma_{1}}\left|u_{t}(t)\right|^{m-2}u_{t}(t)\phi
=\int_{\Omega}\left|u(t)\right|^{p-2}u(t)\phi. \label{eq:def2.1}
\end{equation}
Moreover we say that $u$ is a weak solution of problem \eqref{Pmod}
in $[0,T)\times\Omega$ if $u$ is  a weak solution of \eqref{Pmod} in
$[0,T']\times\Omega$ for all $T'\in (0,T)$.
\end{definition}

\begin{rem}
Since $p\leq 1+2^*/2$ and $\varphi\in H^{1}_{\Gamma_{0}}(\Omega)$
the integral in the right--hand side of (\ref{eq:def2.1}) makes
sense due to the Sobolev Embedding Theorem.
\end{rem}

\begin{proof}[\bf Proof of Theorem \ref{Thloc}] We set, for any $0<T<\infty$, the Banach space
    $Y_{T}=C(\left[0,T\right];H^{1}_{\Gamma_{0}}(\Omega))$
endowed with the usual norm $\left\|u\right\|_{Y_{T}}=\left\|u\right\|_{L^{\infty}(0,T;H^{1}_{\Gamma_{0}}(\Omega))}$, and the closed convex set
    $X_{T}=\left\{u\in Y_{T}: u(0)=u_{0}\right\}$.
Let  $u\in X_{T}$. By $(\ref{eq:pr})$ we have $2(p-1)\leq 2^*$ and
then, by the Sobolev Embedding Theorem,
\begin{equation}
\|u(t)\|_{2(p-1)}\leq K_{0}\left\|u(t)\right\|_{H^{1}_{\Gamma_{0}}},
    \quad \forall t\in [0,T],\label{zidane}
\end{equation}
for some $K_0=K_0(\Omega)>0$ (in the sequel of the proof $K_i$,
$i\in\N$, will denote suitable positive constants depending on $p$,
$n$ and $\Omega$). Hence $\left|u\right|^{p-2}u\in
L^{\infty}((0,T);L^{2}(\Omega))$. Then by Theorem \ref{eq:T1} there
is a unique weak solution $v$ of the problem
\begin{equation}
\begin{cases}
v_{t}-\Delta v=\left|u\right|^{p-2}u & \mbox{in } (0,T)\times\Omega,\\
v=0 & \mbox{on } \left[0,T\right)\times\Gamma_{0},\\
\frac{\partial v}{\partial\nu} = -\left|v_{t}\right|^{m-2}v_{t} & \mbox{on } \left[0,T\right)\times\Gamma_{1},\\
v(0,x)=u_{0}(x) & \mbox{in } \Omega.
\end{cases}
\label{eq:P4}
\end{equation}
Moreover
    $v\in C(\left[0,T\right];H^{1}_{\Gamma_{0}}(\Omega))
, v_{t}\in L^{m}((0,T)\times\Gamma_{1})\cap
L^{2}((0,T)\times\Omega)$
 and the energy identity
    \begin{equation}\label{stellarossabelgrado}
    \frac{1}{2}\left\|\nabla v\right\|^{2}_{2}\Big| ^{t}_{0}+\int^{t}_{0}\left(\left\|v_{t}\right\|^{m}_{m,\Gamma_{1}}+\left\|v_{t}\right\|^{2}_{2}\right)=\int^{t}_{0}\int_{\Omega}\left|u\right|^{p-2}uv_{t}
\end{equation}
holds for all $t\in \left[0,T\right]$. We define $\Phi :
X_{T}\rightarrow X_{T}$ by $\Phi(u)=v$, where $v$ denotes the
solution of $(\ref{eq:P4})$ that corresponds to $u$. We are going to
prove that we can apply the Banach Contraction Theorem to $\Phi :
B_{R}\rightarrow B_{R}$ where $B_{R}=\left\{u\in X_{T}:
\left\|u\right\|_{Y_{T}}\leq R\right\}$, provided that $R$ is
sufficiently large and $T$ is sufficiently small. Note that $B_R$ is
non--empty for
\begin{equation}
R\geq R_{0}:=\left\|u_{0}\right\|_{H^{1}_{\Gamma_{0}}}.
\label{eq:r0}
\end{equation}

We first claim that $\Phi$ maps $B_{R}$ into itself for $R$
sufficiently large and $T$ small enough. Let $u\in B_{R}$. By
\eqref{stellarossabelgrado} and \eqref{zidane} we get, for
$t\in\left[0,T\right]$,
\begin{equation*}
\frac{1}{2}\left\|\nabla v(t)\right\|^{2}_{2}+
\int^{t}_{0}\left\|v_{t}\right\|^{2}_{2} \leq
\frac{1}{2}\left\|\nabla u_{0}\right\|^{2}_{2}+ K_0
^{2(p-1)}\int^{t}_{0}\left\|u\right\|^{p-1}_{H^1_{\Gamma_0}}\left\|v_{t}\right\|_{2}.
\end{equation*}
Now using Young inequality it follows that, for all
$t\in\left[0,T\right]$,
\begin{equation*}
\begin{aligned}
\frac 12\left\|\nabla
v(t)\right\|^{2}_{2}+\int^{t}_{0}\left\|v_{t}\right\|^{2}_{2}
\leq & \frac{1}{2}\left\|\nabla u_{0}\right\|^{2}_{2}+ K_{1}R^{p-1}\int^{t}_{0}\left\|v_{t}\right\|_{2}\\
\leq & \frac{1}{2}\left\|\nabla
u_{0}\right\|^{2}_{2}+\frac{1}{2}K_{1}^{2}R^{2(p-1)}
T+\frac{1}{2}\int^{t}_{0}\left\|v_{t}\right\|^{2}_{2}.
\end{aligned}
\end{equation*}
Hence
\begin{equation}\label{Coco52} \frac{1}{2}\left\|\nabla
v(t)\right\|^{2}_{2}+\frac{1}{2}\int^{t}_{0}\left\|v_{t}\right\|^{2}_{2}
\leq  \frac{1}{2}\left\|\nabla u_{0}\right\|^{2}_{2}+K_{2}R^{2(p-1)}
T.
\end{equation}
Consequently, by \eqref{eq:r0},
\begin{align}
&\left\|\nabla v\right\|^{2}_{L^{\infty}(0,T;L^{2}(\Omega))}  \leq
 R^{2}_{0}+2K_{2}R^{2(p-1)}T \label{eq:3.n}\\
\intertext{and} &\int^{T}_{0}\left\|v_{t}\right\|^{2}_{2}  \leq
R^{2}_{0}+2K_{2}R^{2(p-1)}T. \label{eq:3.1}
\end{align}
Using H\"{o}lder inequality we have $\left\|v(t)\right\|_{2}  =
\left\|u_{0}+\int^{t}_{0}v_{t}(s)ds\right\|_{2}
 \leq
\left\|u_{0}\right\|_{2}+T^{\frac{1}{2}}\left(\int^{t}_{0}\left\|v_{t}\right\|^{2}_{2}\right)^{\frac{1}{2}}
$
and so, by \eqref{eq:3.1},
\begin{equation}
\left\|v(t)\right\|^{2}_{2}  \leq  2\left\|u_{0}\right\|^{2}_{2}+ 2
T\left(\int^{t}_{0}\left\|v_{t}\right\|^{2}_{2}\right) \leq
2(1+T)R^{2}_{0}+4K_{2}R^{2(p-1)}T^{2}.\label{eq:3.2}
\end{equation}
Now restricting to $T\leq 1$ we have $T^{2}\leq T$ and so combining $(\ref{eq:3.n})$ and $(\ref{eq:3.2})$ we get
\begin{equation}
\left\|v\right\|^{2}_{Y_{T}}  \leq (3+2T)R^{2}_{0}+6K_{2}R^{2(p-1)}T
 \leq  5R^{2}_{0}+6K_{2}R^{2(p-1)}T.\label{eq:lastr}
\end{equation}
By $(\ref{eq:lastr})$ in order to prove that $v\in B_{R}$, it is
enough to show that $5R^{2}_{0}\leq \frac{1}{2}R^{2}$ and
$6K_{2}R^{2(p-1)}T\leq\frac{1}{2}R^{2}$. Hence our claim holds for
\begin{equation}
R=4R_{0}\quad\text{and}\quad T \leq
\min\left\{1,K_{3}R^{2(2-p)}_{0}\right\}. \label{eq:T}
\end{equation}
In the sequel we shall assume that (\ref{eq:T}) holds.

We now claim that, for $T$ small enough, the map $\Phi$ is a
contraction. Let $u,\bar{u}\in B_{R}$, and denote $v=\Phi(u)$,
$\bar{v}=\Phi(\bar{u})$, $w=v-\bar{v}$. Clearly, $w$ is a weak
solution (in the sense of Lemma~\ref{eq:l1}) of the problem
\begin{equation}
\mbox{
$\left\{\begin{array}{llll}
$$w_{t}-\Delta w=\left|u\right|^{p-2}u-\left|\bar{u}\right|^{p-2}\bar{u} & \mbox{in } (0,T)\times\Omega,$$\\
$$w=0 & \mbox{on } \left[0,T\right)\times\Gamma_{0},$$\\
$$\frac{\partial w}{\partial\nu} = -\left|v_{t}\right|^{m-2}v_{t}+\left|\bar{v}_{t}\right|^{m-2}\bar{v}_{t}
& \mbox{on } \left[0,T\right)\times\Gamma_{1},$$\\
$$w(0,x)=0 & \mbox{in } \Omega.$$
\end{array}
\right.$}\label{eq:P6}
\end{equation}
Since $v_{t},\bar{v}_{t}\in L^{m}((0,T)\times\Gamma_{1})$, we also
know that $\left|v_{t}\right|^{m-2}v_{t}$ and
$\left|\bar{v}_{t}\right|^{m-2}\bar{v}_{t}$ belong to
$L^{m'}((0,T)\times\Gamma_{1})$. Moreover, by $(\ref{eq:pr})$, the
functions $\left|u\right|^{p-2}u$ and
$\left|\bar{u}\right|^{p-2}\bar{u}$ belong to
$L^{2}((0,T)\times\Omega)$. Then we can apply Lemma \ref{eq:l1} so
that, for $t\in [0,T]$,
\begin{multline}
\frac{1}{2}\left\|\nabla w(t)\right\|^{2}_{2}+
\int^{t}_{0}\left\|w_{t}\right\|^{2}_{2}+
\int^{t}_{0}\int_{\Gamma_{1}}\left[\left|v_{t}\right|^{m-2}v_{t}-\left|\bar{v}_{t}\right|^{m-2}\bar{v}_{t}\right]w_t
\\
=\int^{t}_{0}\int_{\Omega}\left[\left|u\right|^{p-2}u-\left|\bar{u}\right|^{p-2}\bar{u}\right]w_{t}.
\end{multline}

Using the monotonicity of the map $x \rightarrow
\left|x\right|^{m-2}x$ and the elementary inequality
    \begin{equation}\left|\left|A\right|^{p-2}A - \left|B\right|^{p-2}B\right|\leq K_{4}\left|A -
     B\right|\left(\left|A\right|^{p-2} +
     \left|B\right|^{p-2}\right),\label{**ester}
\end{equation}
for $A,B\in\mathbb{R}$, $p\geq2$, we get
\begin{equation}\label{spinoza}
\frac{1}{2}\left\|\nabla w(t)\right\|^{2}_{2}+
\int^{t}_{0}\left\|w_{t}\right\|^{2}_{2}\leq
K_{4}\int^{t}_{0}\int_{\Omega}
\left(\left|u\right|^{p-2}+\left|\bar{u}\right|^{p-2}\right)
|u-\bar{u}|\,|w_t|.
\end{equation}
We now set $r=2^*$ if $n\in\N$, $n\not=2$, while $r=2p$ when $n=2$,
so that $2\le p\le 1+r/2\le 1+2^*/2$ and $r>2$. We also fix $s
> 2$ such that $\frac{1}{s}+\frac{1}{r}+\frac{1}{2}=1$, that is
$s=\frac{2r}{r-2}$. By applying triple H\"{o}lder inequality and the
elementary inequality
\begin{equation}
\left(A +
B\right)^{\tau}\leq\max\left\{1,2^{\tau-1}\right\}\left(A^{\tau} +
B^{\tau}\right)\quad \mbox{for  } A,B\geq 0,\quad \tau\geq 0,
\label{eq:3.4}
\end{equation}
from \eqref{spinoza} we get
\begin{equation*}
\frac{1}{2}\left\|\nabla w(t)\right\|^{2}_{2}+
\int^{t}_{0}\!\!\left\|w_{t}\right\|^{2}_{2}  \leq
K_{5}\int^{t}_{0}\left(\int_{\Omega}\left(\left|u\right|^{s(p-2)}+\left|\bar{u}\right|^{s(p-2)}\right)\right)^{\frac{1}{s}}
\left\|u-\bar{u}\right\|_{r}\left\|w_{t}\right\|_{2}.
\end{equation*}
But $s(p-2)\leq r$ since $p\leq 1+\frac{r}{2}$, so by the Sobolev
Embedding Theorem and weighted Young inequality we obtain, for any
$\varepsilon > 0$,
\begin{equation}
\begin{aligned}
\frac{1}{2}\left\|\nabla w(t)\right\|^{2}_{2}+
\int^{t}_{0}\left\|w_{t}\right\|^{2}_{2}  \leq &
K_{6}\int^{t}_{0}\left[\left\|u\right\|^{s(p-2)}_{r}+
\left\|\bar{u}\right\|^{s(p-2)}_{r}\right]^{\frac{1}{s}}
\left\|u-\bar{u}\right\|_{r}\left\|w_{t}\right\|_{2}\\
 \leq &  K_{7}R^{p-2}\int^{t}_{0}
\left\|u-\bar{u}\right\|_{r}\left\|w_{t}\right\|_{2}\\
 \leq &  K_{8}R^{p-2}\int^{t}_{0}
\left\|u-\bar{u}\right\|_{H^{1}_{\Gamma_{0}}}\left\|w_{t}\right\|_{2}\\
\leq &  K_{8}R^{p-2}\left[\frac{1}{2\varepsilon}\int^{t}_{0}
\left\|u-\bar{u}\right\|^{2}_{H^{1}_{\Gamma_{0}}}+\frac{\varepsilon}{2}\int^{t}_{0}\left\|w_{t}\right\|^{2}_{2}\right]
 \end{aligned}\label{crunccrunc}
\end{equation}
and consequently
$$\frac{1}{2}\left\|\nabla w(t)\right\|^{2}_{2}+
\int^{t}_{0}\left\|w_{t}\right\|^{2}_{2} \leq
\frac{K_{9}R^{p-2}}{\varepsilon}T
\left\|u-\bar{u}\right\|^{2}_{L^{\infty}(0,T;H^{1}_{\Gamma_{0}}(\Omega))}+
K_{9}R^{p-2}\varepsilon\int^{t}_{0}\left\|w_{t}\right\|^{2}_{2}.$$
Now we choose $\varepsilon=1/(2K_{9}R^{p-2})$ so  previous estimate
reads as
\begin{equation}
\frac{1}{2}\left\|\nabla w(t)\right\|^{2}_{2}+
\frac{1}{2}\int^{t}_{0}\left\|w_{t}\right\|^{2}_{2}\leq
2K^{2}_{9}R^{2(p-2)}T\left\|u-
\bar{u}\right\|^{2}_{L^{\infty}(0,T;H^{1}_{\Gamma_{0}}(\Omega))},
\label{eq:f4.12}
\end{equation}
and consequently
\begin{align} \left\|\nabla w(t)\right\|_{2}\leq
&K_{10}R^{p-2}\sqrt{T}\left\|u-\bar{u}\right\|_{L^{\infty}(0,T;H^{1}_{\Gamma_{0}}(\Omega))}\quad\text{for
all $t\in [0,T]$} \label{eq:c1}\\ \intertext{and}
\int^{T}_{0}\left\|w_{t}\right\|^{2}_{2}\leq &
K_{11}R^{2(p-2)}T\left\|u-\bar{u}\right\|^{2}_{L^{\infty}(0,T;H^{1}_{\Gamma_{0}}(\Omega))}.
\label{eq:c3}
\end{align}
Then, since $w(0)=0$, by H\"{o}lder inequality and \eqref{eq:c3}
\begin{equation}
\left\|w(t)\right\|_{2}  \le
\int^{t}_{0}\!\!\left\|w_{t}\right\|_{2}
 \leq  T^{\frac{1}{2}}\left(\int^{t}_{0}\!\!\left\|w_{t}\right\|^{2}_{2}\right)^{\frac{1}{2}}
 \leq  T
K_{12}R^{p-2}\left\|u-\bar{u}\right\|_{L^{\infty}(0,T;H^{1}_{\Gamma_{0}}(\Omega))}.
 \label{eq:c2}
\end{equation}
By combining (\ref{eq:c1}) and (\ref{eq:c2}) we consequently get (as
$T\leq 1$)
\begin{equation}
\left\|w(t)\right\|^{2}_{H^{1}_{\Gamma_{0}}}  \leq
K^{2}_{13}R^{2(p-2)}T\left\|u-\bar{u}\right\|^{2}_{L^{\infty}(0,T;H^{1}_{\Gamma_{0}}(\Omega))}.
\label{3crunc}\end{equation}
 Then $\Phi$ is a contraction provided $K_{13}R^{p-2}\sqrt T<1$,
that is, by \eqref{eq:T},  provided
\begin{equation}T <K_{13}^{-2}(4
R_0)^{2(2-p)}. \label{4crunc}\end{equation} We can finally choose
$T^{*}=\min\left\{1,K_{3}R^{2(2-p)}_{0},\frac{1}{2}K^{-2}_{13}(4R_{0})^{2(2-p)}\right\}$
which is decreasing in $R_{0}$. So, by applying Banach Contraction
Theorem with $T=T^*$, there is a weak solution of \eqref{Pmod} on
$[0,T^*]\times \Omega$ satisfying \eqref{eq:4.12}--\eqref{eq:ei2}.
Moreover \eqref{eq:gengis} follows by \eqref{eq:T}.

In order to prove that the solution is unique we use a standard
procedure of ODEs, using previous claims, which is briefly outlined
as follows. Let $u$, $\widetilde{u}$ be two weak solutions of
$(\ref{Pmod})$ on $[0,T^{*}]\times\Omega$. By Lemma \ref{eq:l1} we
have $u, \widetilde{u}\in
C(\left[0,T^{*}\right];H^{1}_{\Gamma_{0}}(\Omega))$. Suppose by
contradiction that $u\neq\widetilde{u}$. Then
\begin{equation}
T'=\sup\left\{\tau>0:u=\widetilde{u}\mbox{ on
}[0,\tau]\right\}<T^{*}\quad\text{and $u(T')=\widetilde{u}(T')$ by
continuity.} \label{eq:T'}
\end{equation}
Setting $u_{1}(t)=u(t+T')$,
$\widetilde{u}_{1}(t)=\widetilde{u}(t+T')$ we have $u_1,
\widetilde{u}_1\in
C(\left[0,T^{*}-T'\right];H^{1}_{\Gamma_{0}}(\Omega))$ and
$\widetilde{u}_{0}:=u_{1}(0)=\widetilde{u}_{1}(0)$. Then $u_{1}$,
$\widetilde{u}_{1}$ are weak solutions of $(\ref{Pmod})$ with
initial datum $\widetilde{u}_{0}$. By continuity there is $0<T''\leq
T^{*}-T'$ such that
$$\max\{\left\|u_{1}\right\|_{C(\left[0,T''\right];H^{1}_{\Gamma_{0}}(\Omega))},
\left\|\widetilde{u}_{1}\right\|_{C(\left[0,T''\right];H^{1}_{\Gamma_{0}}(\Omega))}\}\leq
4\left\|\widetilde{u}_{0}\right\|_{H^{1}_{\Gamma_{0}}}.$$ Hence
$u_1$ and $\widetilde{u}_{1}$ are fixed points for $\Phi$ in
$B_{4\|\widetilde{u}_0\|_{H^1_{\Gamma_0}}}$ when $T=T''$, so by
previous claim  $u_{1}=\widetilde{u}_{1}$ on $[0,T]$, contradicting
(\ref{eq:T'}).
\end{proof}
\begin{proof}[\bf Proof of Theorem \ref{Thalt}]
The existence of the unique maximal solution $u$ of \eqref{Pmod}
follows by Theorem~\ref{Thloc} in a standard way: first one  sets
$\cal{U}$ to be the set of all tweak solutions of \eqref{Pmod}, then
one proves that any two elements of $\cal{U}$ must coincide on the
intersection of their domains,  arguing as at the end of previous
proof, finally one defines $u(t)$ to coincide with any of these
solution for $t$ in the union of the domains.

Next, in order to prove that the alternative (i)--(ii) holds, let us
suppose, by contradiction, that
\begin{equation}\label{alternative} T_{\max} < \infty \quad
\mbox{and  } \varlimsup_{t\rightarrow
T^{-}_{\max}}\left\|u(t)\right\|_{H^{1}_{\Gamma_{0}}} <  \infty.
\end{equation}
 Then there is a sequence $T_{n}\rightarrow
T^{-}_{\max}$ such that $\left\|u(T_n)\right\|_{H^{1}_{\Gamma_{0}}}$
is bounded. Thus, by Theorem $\ref{Thloc}$, the Cauchy problem
$(\ref{Pmod})$ with initial time $T_{n}$ and initial datum $u(T_n)$
as a unique weak solution in $[T_n,T_n+T']$,  where
$T'=T^*(\sup_{n\in\N}\|u(T_n)\|_{H^1_{\Gamma_0}},
m,p,\Omega,\Gamma_1)$ is independent on $n$. This leads to a
contradiction, since, in this way, we can continue the solution to
the right of $T_{\max}$.

Now, in order to prove that $u$ depends continuously on the initial
datum, we fix $T\in (0,T_{\max})$ and we denote
$M=\left\|u\right\|_{C(\left[0,T\right];H^{1}_{\Gamma_{0}}(\Omega))}$.
Since $u_{0n}\rightarrow u_{0}$ in $H^{1}_{\Gamma_{0}}(\Omega)$
there is $n_{1}\in\N$ such that
    $\left\|u_{0n}\right\|_{H^{1}_{\Gamma_{0}}}\leq \left\|u_{0}\right\|_{H^{1}_{\Gamma_{0}}}+1 \leq M+1$.
Then, by Theorem \ref{Thloc}, problem $(\ref{Pmod})$ with initial
datum $u_{0n}$ has an unique solution $u^{n}$ in
$[0,T^{*}]\times\Omega$, with $T^{*}=T^{*}(M+1, m, p, \Omega,
\Gamma_{1})\in(0,1]$ and
\begin{equation}
\left\|u^{n}\right\|_{C(\left[0,T^{*}\right];H^{1}_{\Gamma_{0}}(\Omega))}\leq
4\left\|u_{0n}\right\|_{H^{1}_{\Gamma_{0}}} \leq 4(M+1)
\label{eq:M+1}
\end{equation}
for all $n\in\mathbb{N}$. Now we define $w^n=u^n-u$, which is a weak
solution of the problem
\begin{equation*}
\mbox{
$\left\{\begin{array}{llll}
$$w^n_{t}-\Delta w^n=\left|u^n\right|^{p-2}u^n-\left|u\right|^{p-2}u & \mbox{in } (0,T^{*})\times\Omega,$$\\
$$w^n=0 & \mbox{on } \left[0,T^{*}\right)\times\Gamma_{0},$$\\
$$\frac{\partial w^n}{\partial\nu} = -\left|u^n_{t}\right|^{m-2}u^n_{t}+\left|u_{t}\right|^{m-2}u_{t} & \mbox{on } \left[0,T^{*}\right)\times\Gamma_{1},$$\\
$$w^n(0)=u_{0n}-u_{0} & \mbox{in } \Omega$$
\end{array}
\right.$}
\end{equation*}
in the sense of Lemma~\ref{eq:l1}. Consequently
\begin{multline}
\frac{1}{2}\left\|\nabla w^n\right\|^{2}_{2} \Big |
^{t}_{0}+\int^{t}_{0}\left\|w^n_t\right\|^{2}_{2}
+\int^{t}_{0}\int_{\Gamma_1}\left [
|u_t^n|^{m-2}u_t^n-|u_t|^{m-2}u_t\right
]w^n_t \\
=\int^{t}_{0}\int_{\Omega}\left[\left|u^{n}\right|^{p-2}u^{n}-\left|u\right|^{p-2}u\right]w^n_{t}.
\label{eq:rhs}
\end{multline}
Then, keeping the notation of  the proof of Theorem \ref{Thloc} and
using the arguments already used to prove \eqref{crunccrunc}
together with \eqref{eq:M+1} we get the estimate
\begin{multline}\label{fascio}
\frac{1}{2}\left\|\nabla w^n(t)\right\|^{2}_{2} +
\int^{t}_{0}\|w^n_t\|^{2}_{2} \\\leq
4^{p-2}(M+1)^{p-2}K_8\left[\frac{1}{2\varepsilon}\int^{t}_{0}\|w^n\|^{2}_{H^{1}_{\Gamma_{0}}}+\frac{\varepsilon}{2}
\|w^n_t\|^{2}_{2}\right]+
\frac{1}{2}\left\|\nabla(u_{0n}-u_{0})\right\|^{2}_{2}
\end{multline}
for any $\varepsilon
> 0$.  Consequently, for $\varepsilon > 0$ sufficiently  small we have
\begin{equation}
\frac{1}{2}\left\|\nabla w^n(t)\right\|^{2}_{2}+
\frac{1}{2}\int^{t}_{0}\|w^n_t\|^{2}_{2}\leq C_3
\int^{t}_{0}\left\|w^n\right\|^{2}_{H^{1}_{\Gamma_{0}}}+\frac{1}{2}\left\|\nabla(u_{0n}-u_{0})\right\|^{2}_{2}
\label{eq:3.5}
\end{equation}
where $C_3=C_3(p,n,\Omega,u_0,T)>0$. Moreover, since $T^*\le 1$, by
using H\"{o}lder inequality we get
$\left\|w^n(t)\right\|_{2}\leq\left\|u_{0n}-u_{0}\right\|_{2}+\left(\int^{t}_{0}\|w^n_t\|^2_{2}\right)^{1/2}$
and so by \eqref{eq:3.5}
\begin{equation}
\left\|w^n(t)\right\|^{2}_{2} \leq
2\left\|u_{0n}-u_{0}\right\|^{2}_{2}+2\int^{t}_{0}\|w^n_t\|^{2}_{2}
\leq
2\left\|u_{0n}-u_{0}\right\|^{2}_{2}+4C_3\int^{t}_{0}\|w^{n}\|^{2}_{H^{1}_{\Gamma_{0}}}.
\label{eq:wn}
\end{equation}
Combining $(\ref{eq:3.5})$ and $(\ref{eq:wn})$ we get
\begin{equation}
\left\|w^n(t)\right\|^{2}_{H^{1}_{\Gamma_{0}}}\leq
2\left\|u_{0n}-u_{0}\right\|^{2}_{H^{1}_{\Gamma_{0}}}+C_4\int^{t}_{0}\|w^n\|^{2}_{H^{1}_{\Gamma_{0}}}
\label{eq:Gron}
\end{equation}
where $C_4=C_4(p,n,\Omega,u_0,T)>0$. By Gronwall inequality  the
estimate
\begin{equation}
\left\|w^n(t)\right\|_{H^{1}_{\Gamma_{0}}}\leq\sqrt{2}\left\|u_{0n}-u_{0}\right\|_{H^{1}_{\Gamma_{0}}}e^{\frac{C_4}{2}t},\quad\text{for
all $t\in [0,T^*]$,}\label{kjh}
\end{equation}
follows. In particular we have
\begin{equation}
\left\|w^n(T^{*})\right\|_{H^{1}_{\Gamma_{0}}}\leq\sqrt{2}\left\|u_{0n}-u_{0}\right\|_{H^{1}_{\Gamma_{0}}}e^{\frac{C_4}{2}T^{*}}.
\label{eq:particular}
\end{equation}
Then, since $u_{0n}\to u_0$ as $n\to\infty$,  for $n\geq n_{2}$,
with $n_{2}$ sufficiently large, we have
$\left\|u^n(T^{*})\right\|_{H^{1}_{\Gamma_{0}}}\leq\left\|u(T^{*})\right\|_{H^{1}_{\Gamma_{0}}}+1\leq
M+1$. Hence we get that $u^n$ is defined in $[T^*,2T^*]$. Moreover,
by repeating previous argument for $t\in \left[T^{*},2T^{*}\right]$
and using \eqref{eq:particular}, we get
$\left\|w^n(t)\right\|_{H^{1}_{\Gamma_{0}}}  \leq
\sqrt{2}\left\|w^n(T^{*})\right\|_{H^{1}_{\Gamma_{0}}}e^{\frac{C_4}{2}(t-T^{*})}
 \leq
2\left\|u_{0n}-u_{0}\right\|_{H^{1}_{\Gamma_{0}}}e^{\frac{K_{15}}{2}t}$.
After a finite number $k=\left[ \frac{T}{T^{*}}\right]$ of
iterations we get that for $n$ large enough  $u^n$ is defined in
$\left[0,T\right]$ and
$\left\|u^n(t)-u(t)\right\|_{H^{1}_{\Gamma_{0}}}\leq2^{\frac{k}{2}}\left\|u_{0n}-u_{0}\right\|_{H^{1}_{\Gamma_{0}}}
e^{\frac{C_4}{2}t}$ for $t\in\left[0,T\right]$, concluding the
proof.
\end{proof}

\section{Proofs of
Theorems~\ref{Thglo}~\and\ref{Thblow}.}\label{section4}

When $\sigma(\Gamma_{0}) > 0$ a Poincar\`e type inequality holds
(see \cite[Corollary~4.5.3]{ziemer}) and we can take $\left\|\nabla
u\right\|_{2}$ as an equivalent norm in
$H^{1}_{\Gamma_{0}}(\Omega)$. Then using the Sobolev's Embedding
Theorem, since $p\leq 1+2^*/2\leq 2^*$, we have
\begin{equation}
B_{1}:=\sup_{u\in H^{1}_{\Gamma_{0}}(\Omega), u\neq
0}\frac{\left\|u\right\|_{p}}{\left\|\nabla u\right\|_{2}} <
+\infty. \label{eq:B1}
\end{equation}
We denote, when $2<p\leq 1+2^*/2$,
\begin{align}\label{lambdais}
    & \lambda_{1}=B^{-\frac{p}{p-2}}_{1},\qquad
    \widetilde{\lambda_1}=B_1^{-\frac{2}{p-2}}, \qquad
    E_{1}=\left(\frac{1}{2}-\frac{1}{p}\right)\lambda^{2}_{1},\\
    & W_{1}=\left\{u_{0}\in H^{1}_{\Gamma_{0}}\left(\Omega\right): J\left(u_{0}\right)<E_{1}\quad{and}\quad\left\|\nabla u_{0}\right\|_{2}<
    \lambda_{1}\right\},\label{W1}\\
    &\widetilde{W_{1}}=\left\{u_{0}\in H^{1}_{\Gamma_{0}}\left(\Omega\right): J\left(u_{0}\right)<E_{1}\quad{and}\quad\left\|u_{0}\right\|_p<
    \widetilde{\lambda_{1}}\right\},\label{W1tilde}\\
    &W_{2}=\left\{u_{0}\in H^{1}_{\Gamma_{0}}\left(\Omega\right):
    J\left(u_{0}\right)<E_{1}\quad{and}\quad\left\|\nabla u_{0}\right\|_{2}>
    \lambda_{1}\right\},\label{W2}\\
    \intertext{and}
    &\widetilde{W_{2}}=\left\{u_{0}\in H^{1}_{\Gamma_{0}}\left(\Omega\right):
    J\left(u_{0}\right)<E_{1}\quad{and}\quad\left\|u_{0}\right\|_p>
    \widetilde{\lambda_{1}}\right\}.\label{W2tilde}
    \end{align}

\noindent At first we give the following useful characterization of
$d$, $W_s$ and $W_u$.
\begin{lem} \label{eq:Ld}
Suppose $2<p\leq 1+2^*/2$, $\sigma(\Gamma_{0})>0$ and let $d$,
$W_{s}$ and  $W_{u}$ be respectively defined by \eqref{eq:d},
\eqref{eq:Ws} and \eqref{eq:Wu}. Then $E_{1}=d$,
$W_s=W_1=\widetilde{W_1}$ and $W_u=W_2=\widetilde{W_2}$.
\end{lem}

\begin{proof}
An easy calculation shows that for any $u\in
H^{1}_{\Gamma_{0}}(\Omega)\setminus\{0\}$  we have
    $\max\limits_{\lambda > 0}J(\lambda u)=J(\lambda(u)u)=\left(\frac{1}{2}-\frac{1}{p}\right)\left(\frac{\left\|\nabla u\right\|_{2}}{\left\|u\right\|_{p}}
    \right)^{2p/(p-2)}$
where
    $\lambda(u)=\frac{\left\|\nabla u\right\|^{2/(p-2)}_{2}}{\left\|u\right\|^{p/(p-2)}_{p}}$.
Hence, by \eqref{eq:B1}, $d=E_{1}$. In order to show that
$W_s=W_1=\widetilde{W_1}$ we first prove that $W_{s}\subseteq
W_{1}$. Let $u_{0}\in W_{s}$ and suppose, by contradiction, that
$\left\|\nabla u_{0}\right\| \geq \lambda_{1}$. Since $J(u_{0}) < d
=E_{1}$ and $\left\|u_{0}\right\|^{p}_{p}\leq \left\|\nabla
u_{0}\right\|^{2}_{2}$ it follows that
    \[E_{1}>J(u_0) \geq\left(\frac{1}{2}-\frac{1}{p}\right)\left\|\nabla u_{0}\right\|^{2}_{2} \geq \left(\frac{1}{2}-\frac{1}{p}\right)\lambda^{2}_{1},
\]
which contradicts \eqref{lambdais}. By \eqref{eq:B1}, since
$\widetilde{\lambda_1}=B_1\lambda_1$, one immediately gets that
$W_1\subseteq\widetilde{W_1}$. To prove that
$\widetilde{W_1}\subseteq W_{s}$, let $u_{0}\in \widetilde{W_1}$. By
\eqref{eq:B1}, \eqref{W1tilde} and \eqref{lambdais} we have
    $\left\| u_{0}\right\|^{p}_{p}<\widetilde{\lambda_1}^{p-2}\left\|u_{0}\right\|^2_p
     =B_1^{-2}\left\|u_{0}\right\|^2_p\leq\left\|\nabla u_{0}\right\|^{2}_{2}$
and so $K(u_{0})\geq 0$.

In order to show that $W_u=W_2=\widetilde{W_2}$ we first prove that
$W_2\subseteq W_u$. Let $u_{0}\in W_{2}$ and suppose, by
contradiction, that $K(u_{0})> 0$. So
$\left\|u_{0}\right\|^{p}_{p}<\left\|\nabla u_{0}\right\|^{2}_{2}$
by \eqref{eq:K}. Moreover, $J(u_{0}) < d =E_{1}$ and $\left\|\nabla
u_{0}\right\|_{2}> \lambda_{1}$. Then it follows that
    \[E_{1}>\left(\frac{1}{2}-\frac{1}{p}\right)\left\|\nabla u_{0}\right\|^{2}_{2} >\left(\frac{1}{2}-\frac{1}{p}\right)\lambda^{2}_{1},
\]
which contradicts \eqref{lambdais}. By \eqref{eq:B1} one immediately
gets that $\widetilde{W_2}\subseteq W_2$. To prove that
$W_u\subseteq \widetilde{W_2}$ and conclude the proof, we take
$u_{0}\in W_{u}$. We note that, by \eqref{eq:B1}, we have $J(v)\geq
h(\|v\|_p)$ for all $v\in H^1_{\Gamma_0}(\Omega)$, where $h$ is
defined by $h(\lambda)=\frac 12 B_1^{-2}\lambda^2-\frac 1p
\lambda^p$ for $\lambda\geq 0$.  Moreover one easily verify that
$h(\widetilde{\lambda_1})=E_1$. The, since $J(u_0)<E_1$, we have
$\|u_0\|_p\not=\widetilde{\lambda_1}$. Moreover, since $K(u_0)\le
0$, by \eqref{eq:B1} we have
    $B_1^{-2}\left\|u_{0}\right\|^{2}_p\leq \left\|\nabla u_{0}\right\|^2_2\leq
    \left\|u_{0}\right\|^{p}_p$
and so
    $\left\|u_{0}\right\|_p\geq B^{-p/(p-2)}_{1}=\widetilde{\lambda_1}$,
concluding the proof.
\end{proof}
In what follows we shall use the following derivation formula, which
is proved here for the sake of completeness only.
\begin{lem}\label{auxiliary}
Under the assumptions of Theorem~\ref{Thloc}, let $u$ be a weak
solution of problem \eqref{Pmod} in $[0,T]\times\Omega$. Then
\begin{equation}\label{auxiliaryformula}
\frac d{dt}\|u(t)\|_p^p=p\int_\Omega
|u(t)|^{p-2}u(t)u_t(t)\qquad\text{for almost all $t\in (0,T)$.}
\end{equation}
\end{lem}
\begin{proof}
By Definition~\ref{eq:def1}-- (a) and \eqref{eq:pr} we have
$|u|^p\in L^\infty(0,T;L^1(\Omega))$, and consequently $\int_\Omega
|u|^p\in L^\infty(0,T)\subset L^2(0,T)$. It also follows that $u\in
H^1((0,T)\times\Omega )$. Since the real function $x\mapsto |x|^p$
in locally Lipschitz continuous, by the chain rule in Sobolev spaces
(see \cite{mm})  the function  $t\mapsto|u(t,x)|^p$ is absolutely
continous for almost all $x\in\Omega$ and $\frac\partial{\partial t}
|u|^p=p|u|^{p-2}uu_t\in L^2(0,T;L^1(\Omega))\hookrightarrow
L^1((0,T)\times\Omega)$, where assumption \eqref{eq:pr} was used
again. It follows that for all $\varphi\in C^\infty_c(\Omega)$ and
$\chi\in C^\infty_c(0,T)$ we have
$\int_{(0,T)\times\Omega}|u|^p\varphi\chi'=-\int_{(0,T)\times\Omega}p|u|^{p-2}uu_t\varphi\chi$.
Using Fubini's Theorem, since $\varphi$ is arbitrary it follows that
$\int_0^T|u|^p\chi'=-\int_0^T p|u|^{p-2}uu_t\chi$ in $L^1(\Omega)$.
Since $\int_\Omega p|u|^{p-2}uu_t\in L^2(0,T)$ it follows from last
formula that $ \|u\|_p^p\in H^1(0,T)$ and \eqref{auxiliaryformula}
holds in the weak sense. By \cite[Theorem 8.2]{brezis2} we see that
is holds also almost everywhere in $(0,T)$, concluding the proof.
\end{proof}
We now show that $W_s$ and $W_u$ are invariant under the flow
generated by \eqref{Pmod}.
\begin{lem} \label{rem}
Under the assumptions of Theorem~\ref{Thloc}, let $u$ be the weak
maximal solution of problem \eqref{Pmod}. Also assume that
\eqref{eq:sig} holds. Then
\renewcommand{\labelenumi}{(\roman{enumi})}
\begin{enumerate}
    \item  if $u_0\in W_s$ we have $u(t)\in W_s$ for all  $t\in
    [0,T_{\text{max}})$;
    \item  if $u_0\in W_u$ we have $u(t)\in W_u$ for all  $t\in
    [0,T_{\text{max}})$.
\end{enumerate}
\end{lem}
\begin{proof}
By Lemma~\ref{auxiliary}, the energy identity \eqref{eq:ei2} can be
written as
\begin{equation}\label{enid7}
J(u(\tau))\Big |_s^t=-\int_s^t
(\|u_t(\tau)\|_{m,\Gamma_1}^m+\|u_t(\tau)\|_2^2)\,d\tau.
\end{equation}
 Consequently $t\mapsto J(u(t))$ is
decreasing in $[0,T_{\text{max}})$ and by Lemma~\ref{eq:Ld}
\begin{equation}
J(u(t))\leq J(u_{0})<E_1\quad\text{ for all
$t\in[0,T_{\text{max}})$}. \label{eq:et}
\end{equation}
On the other hand, by \eqref{eq:B1} we have the inequality
$J(u(t))\geq g\left(\left\|\nabla u(t)\right\|_{2}\right)$, where
$g(\lambda)=\lambda^{2}/2-B^{p}_{1}\lambda^{p}/p$ for $\lambda\geq
0$. It is straightforward to verify that $g$ is increasing in
$[0,\lambda_1)$ and decreasing in $[\lambda_1,\infty)$, so
$\lambda_1$ is the maximum point for $g$, and that
$g(\lambda_{1})=E_{1}$. Consequently, by \eqref{eq:et} we have
$\|\nabla u(t)\|_2\not=\lambda_1$ for all $t\in[0,T_{\text{max}})$.
Since the function $t\mapsto \|\nabla u(t)\|_{2}$ is continuous, by
Lemma~\ref{eq:Ld} the proof is complete.
\end{proof}

\begin{proof}[\bf Proof of Theorem \ref{Thglo}]
By Theorem \ref{Thalt} we just have to prove that when  $u_0\in W_s$
the alternative \eqref{bu} in Theorem~\ref{Thalt} leads to a
contradiction, which  is  obtained  by combining
Lemma~\ref{rem}--(i) with the Poincar\`e type inequality recalled at
the beginning of the section.
\end{proof}

\begin{proof}[\bf Proof of Theorem \ref{Thblow}] By Theorem \ref{Thalt}
 it is enough to prove that there are no
solutions in the whole $(0,\infty)\times\Omega$. We argue by
contradiction. Since $J(u_{0})<E_{1}$, we can fix
$E_{2}\in(J(u_{0}),E_1)$. We set
\begin{equation}
H(t):=E_{2}-J(u(t)). \label{eq:Ht1}
\end{equation}
By using \eqref{W1}, Lemma~\ref{eq:Ld} and Lemma~\ref{rem} we get
\begin{equation}\label{eq:Ht}
H(t) < E_{1}-\frac{1}{2}\left\|\nabla
u(t)\right\|^{2}_{2}+\frac{1}{p}\left\|u(t)\right\|^{p}_{p}\leq
E_1-\frac 12 \lambda_1^2+\frac{1}{p}\left\|u(t)\right\|^{p}_{p}
\leq\frac{1}{p}\left\|u(t)\right\|^{p}_{p}.
\end{equation}
By \eqref{enid7} we have
\begin{equation}\label{Hprime}
H'(t)=\left\|u_{t}(t)\right\|^{m}_{m,\Gamma_{1}}+\left\|u_{t}(t)\right\|^{2}_{2}\geq
0,
\end{equation}
so that
\begin{equation}
H(t) \geq H(0)=E_{2}-J(u_{0})>0. \label{eq:w6}
\end{equation}
Since, as claimed in Remark~\ref{servo}, it is trivial to verify
that $m_0(p)\leq p$ for $p\ge 2$, by \eqref{eq:pr} and \eqref{eq:w4}
we have $m<1+2^*/2$, which is nothing but the Sobolev critical
exponent for the trace embedding $H^1(\Omega)\hookrightarrow
L^q(\partial\Omega)$ (see \cite[Theorem~5.22, p. 114]{adams}). Hence
we have that $u(t)_{|\Gamma_1}\in L^m(\Gamma_1)$ for all $t\in
[0,T_{\text{max}})$, so we can  take $\phi=u(t)$ in
(\ref{eq:def2.1}). In this way (here and in the sequel of the proof
explicit dependence on $t$ will be omitted) we obtain the identity
\begin{equation}
\left\|u\right\|^{p}_{p}-\left\|\nabla
u\right\|^{2}_{2}=\int_{\Gamma_{1}}\left|u_{t}\right|^{m-2}u_{t}u+(u_{t},u).
\label{eq:w3}
\end{equation}
We estimate the two terms in right-hand side of \eqref{eq:w3}
separately. By H\"{o}lder inequality we get
\begin{equation}
\left|\int_{\Gamma_{1}}\left|u_{t}\right|^{m-2}u_{t}u\right|\leq
\left\|u_{t}\right\|^{m-1}_{m,\Gamma_{1}}\left\|u\right\|_{m,\Gamma_{1}}.
\label{eq:w1}
\end{equation}
To estimate the $L^{m}(\Gamma_{1})$ norm of $u_{|\Gamma_1}$ we first
recall the trace embedding  for Sobolev space of fractional order
(see \cite[Theorem~7.58, p. 218]{adams} and  \cite{unspeskii})
$H^{s}(\R^n)\hookrightarrow W^{\chi,l}(\R^{n-1})$ when $2\le l
<\infty$, $\chi=s-\frac n2 +\frac {n-1}l>0.$ Since
$W^{\chi,l}(\R^{n-1})\hookrightarrow L^l(\R^{n-1})$, using the $C^1$
regularity of $\Omega$ and a standard partition of the unity we have
the trace embedding $H^{s}(\Omega)\hookrightarrow
L^l(\partial\Omega)$ when $2\le l <\infty$, $s-\frac n2 +\frac
{n-1}l>0$ and $0<s\le 1$. Using the last embedding with
$l=\max\{2,m\}$,  the fact that $\partial\Omega$ has finite surface
measure and H\"{o}lder inequality we get
    \begin{equation}\label{es1}
    \left\|u\right\|_{m,\Gamma_{1}}\leq C_1\left\|u\right\|_{H^{s}(\Omega)}
\end{equation}
with $C_1=C_1(m,s,\Omega)>0$, when
\begin{equation}\label{85}
\max \left\{\frac 12,\frac{n}{2}-\frac{n-1}{m} \right\}<s<1.
\end{equation}
Next, by the interpolation inequality (see
\cite[p.49]{lionsmagenes1}
\begin{footnote}{Actually interpolation inequality is stated in the quoted reference only for $C^\infty$
domains $\Omega$, but as explicitly remarked there this assumption
is not optimal. In particular, since $0<s\le 1$, the $C^1$
regularity assumed here is sufficient to prove the result.
Unfortunately the authors were not able to find a reference where
interpolation inequality is stated under  optimal regularity
assumptions. }\end{footnote} ) and the already quoted Poincar\`e
type inequality, we have
    \begin{equation}\label{es2}
    \left\|u\right\|_{H^{s}(\Omega)}\leq C_2\left\|u\right\|^{1-s}_{2}\left\|\nabla u\right\|^{s}_{2}
\end{equation}
$C_2=C_2(s,\Omega,\Gamma_0)>0$. By combining \eqref{es1} and
\eqref{es2} we get
\begin{equation}
\left\|u\right\|_{m,\Gamma_{1}}\leq
C_3\left\|u\right\|^{1-s}_{2}\left\|\nabla u\right\|^{s}_{2}
\label{eq:w2}
\end{equation}
for some $C_3=C_3(m,s,\Omega,\Gamma_0)>0$. By $\eqref{eq:w1}$ and
$(\ref{eq:w2})$
    \begin{equation}\label{458}\left|\int_{\Gamma_{1}}\left|u_{t}\right|^{m-2}u_{t}u\right|\leq C_3\left\|u_{t}\right\|^{m-1}_{\Gamma_{1},m}\left\|u\right\|^{1-s}_{2}\left\|\nabla
    u\right\|^{s}_{2}.
\end{equation}
By weighted Young inequality, if
\begin{equation}\label{86bis}s <
\frac 2m,
\end{equation}
for any $\delta >0$ we have the estimate
    \[\left|\int_{\Gamma_{1}}\left|u_{t}\right|^{m-2}u_{t}u\right|\leq C_3\left[C_4(\delta)\left\|u_{t}\right\|^{m}_{m,\Gamma_{1}}+
    \delta\left\|\nabla u\right\|^{2}_{2} +\delta\left\|u\right\|^{p}_{p}\right]\left\|u\right\|^{1-s-p(1/m-s/2)}_{p}
\]
where $C_4(\delta)=C_4(\delta,m,s)>0$. Consequently, if
$1-s-p\left(\frac 1m -\frac s2\right)<0$, that is if
\begin{equation}\label{87}
s < \left(\frac{p}{m}-1\right)/\left(\frac{p}{2}-1\right),
\end{equation}
setting $\bar{\alpha}_{s}=-\left[1-s-p(1/m-s/2)\right]/p > 0$ we
obtain
\begin{equation}\label{88}
\left|\int_{\Gamma_{1}}\left|u_{t}\right|^{m-2}u_{t}u\right|\leq
C_3\left[C_4(\delta)\left\|u_{t}\right\|^{m}_{m,\Gamma_{1}}+\delta\left\|\nabla
u\right\|^{2}_{2}
+\delta\left\|u\right\|^{p}_{p}\right]\left\|u\right\|^{-p\bar{\alpha}_{s}}_{p}.
\end{equation}
Now we have to show the existence of a value of the parameter $s$
satisfying \eqref{85}, \eqref{86bis} and \eqref{87}. When $1 < m\leq
2$ we have $\frac pm -1>\frac p2-1$ and $\frac 2m >1$, so
\eqref{85}, \eqref{86bis} and \eqref{87} reduce to $\frac 12 <s<1$.
When $m > 2$ we have
    $\left(\frac{p}{m}-1\right)/\left(\frac{p}{2}-1\right)\leq\frac{2}{m}\leq
    1$ and $\frac n2-\frac {n-1}m>\frac 12$,
so \eqref{85}, \eqref{86bis} and \eqref{87} reduce to
    $\frac{n}{2}-\frac{n-1}{m} <s<
    \left(\frac{p}{m}-1\right)/\left(\frac{p}{2}-1\right)$.
Clearly such an $s$ does exist by assumption $(\ref{eq:w4})$. We fix
it.

Now we consider the second term in the right hand side of
\eqref{eq:w3}. Since $p>2$ and $\Omega$ is bounded, applying
H\"{o}lder inequality we easily get
$$|(u_{t},u)|\leq\left\|u_{t}\right\|_{2}\left\|u\right\|_{2}\leq
C_5\left\|u_{t}\right\|_{2}\left\|u\right\|_{p}
    =C_5\left\|u_{t}\right\|_{2}\left\|u\right\|^{\frac{p}{2}}_{p}\left\|u\right\|^{1-\frac{p}{2}}_{p},$$
where $C_5=C_5(\Omega,p)>0$. By weighted Young inequality, for any
$\delta> 0$ we obtain
\begin{equation}\label{89}
\left|\int_{\Omega}u_{t}u\right|\leq C_5\left[\frac
1{4\delta}\left\|u_{t}\right\|^{2}_{2}+\delta\left\|u\right\|^{p}_{p}\right]\left\|u\right\|^{1-\frac{p}{2}}_{p}.
\end{equation}
Now we set
    \begin{equation}\label{parametri}
    \bar{\beta}_{s}=\min\left\{\bar{\alpha}_{s},-\frac{1}{p}+\frac{1}{2}\right\}.
\end{equation}
Since $p>2$ we have $\bar{\beta}_{s}>0$. Since, by \eqref{eq:Ht} and
\eqref{eq:w6} we have
\begin{equation}\label{upbelow}
\|u\|_p\ge [pH(0)]^{1/p} ,
\end{equation} we can combine \eqref{88}
and \eqref{89} (by also using \eqref{parametri}) to obtain
\begin{multline}
\left|\int_{\Gamma_{1}}\left|u_{t}\right|^{m-2}u_{t}u\right|+ |(u_{t},u)|\\
\leq
C_7\left[C_6(\delta)\left(\left\|u_{t}\right\|^{m}_{m,\Gamma_{1}}+\left\|u_{t}\right\|^{2}_{2}\right)+\delta\left\|\nabla
u\right\|^{2}_{2}
+2\delta\left\|u\right\|^{p}_{p}\right]\left\|u\right\|^{-p\bar{\beta}_{s}}_{p},
\label{eq:twoe} \end{multline} where $C_7=C_7(m,p,\Omega, H(0))>0$
and $C_6(\delta)=C_6(\delta,p,m)>0$. By combining \eqref{eq:w3} with
\eqref{eq:twoe} we get
    \begin{equation}\label{104}\left\|u\right\|^{p}_{p}-\left\|\nabla u\right\|^{2}_{2}
    \leq C_{7}\left[C_6(\delta)\left(\left\|u_{t}\right\|^{m}_{m,\Gamma_{1}}+\left\|u_{t}\right\|^{2}_{2}\right)+
    \delta\left\|\nabla u\right\|^{2}_{2} +2\delta\left\|u\right\|^{p}_{p}\right]\left\|u\right\|^{-p\bar{\beta}_{s}}_{p}
\end{equation}
Consequently, by \eqref{Hprime} and \eqref{upbelow}
\begin{equation*}
\begin{aligned}
\left\|u\right\|^{p}_{p}-\left\|\nabla u\right\|^{2}_{2}  \leq &
C_8(\delta) H'(t)\left\|u\right\|^{-p\bar{\beta}_{s}}_{p}
+ C_7(p H(0))^{-\bar{\beta}_s}\delta \left[\|\nabla u\|^{2}_{2}+ 2 \|u\|^{p}_{p}\right]\\
\leq & C_8(\delta)H'(t)\left\|u\right\|^{-p\bar{\beta}_{s}}_{p}+
C_{9}\delta\left[\left\|\nabla u\right\|^{2}_{2}+
\left\|u\right\|^{p}_{p}\right]
\end{aligned}
\end{equation*}
where $C_8(\delta)=C_8(\delta,m,p,H(0),\Omega)>0$ and
$C_9=C_9(m,p,H(0),\Omega)>0$. Consequently
    \[2(1+C_9\delta)\left[-\frac{\left\|\nabla u\right\|^{2}_{2}}{2}\right]+ p(1-C_9\delta)\frac{1}{p}\left\|u\right\|^{p}_{p}
    \leq C_8(\delta)H'(t)\left\|u\right\|^{-p\bar{\beta}_{s}}_{p}.
\]
By \eqref{eq:Ht1} the last estimate can be rewritten as
\begin{multline}
2\left(1+C_9\delta\right)H(t)-2\left(1+C_9\delta\right) E_{2}+ \left[p\left(1-C_9\delta\right)-2\left(1+C_9\delta\right)\right]
\frac{1}{p}\left\|u\right\|^{p}_{p}\\
\leq C_8(\delta)H'(t)\left\|u\right\|^{-p\bar{\beta}_{s}}_{p}.
\label{eq:w5}
\end{multline}
Now, by Lemma~\ref{rem}, \eqref{lambdais} and \eqref{W2tilde} we
have $\|u\|_p^p\ge \widetilde{\lambda_1}^p=\lambda_1^2$, so
previous estimates yields
\begin{multline}
2\left(1+C_9\delta\right)H(t)-2\left(1+C_9\delta\right) E_2+
\lambda_1^2 \left[\left(1-C_9\delta\right)-\tfrac 2p
\left(1+C_9\delta\right)\right]
\\
\leq C_8(\delta)H'(t)\left\|u\right\|^{-p\bar{\beta}_{s}}_{p}.
\label{eq:w666}
\end{multline}
Now, since $E_2<E_1$, using \eqref{lambdais} and the fact that $C_9$
is independent on $\delta$,  as $\delta\rightarrow 0^{+}$ we have
\begin{multline*}
-2(1+C_9\delta)E_2+\lambda_1^2\left[(1-C_9\delta)-\frac
2p(1+C_9)\delta)\right]\to -2E_2+\lambda_1^2\frac {p-2}p\\ >
-2E_1+\lambda_1^2\frac {p-2}p=0
\end{multline*}
Hence, by fixing $\delta>0$ sufficiently small, there exists two
positive constants $C_{10}$ and  $C_{11}$ dependent on $m$, $p$,
$H(0)$ and $\Omega$ such that
    \begin{equation}
    \label{fest}C_{10}H(t)\leq C_{11}H'(t)\left\|u\right\|^{-p\bar{\beta}_{s}}_{p}.
\end{equation}
By \eqref{eq:Ht} the last estimate implies that
    \begin{equation}\label{fest2}H'(t)\geq
C_{12}H^{1+\bar{\beta}_{s}}_{p}(t)
\end{equation}
where $C_{12}=C_{12}(m,p,H(0),\Omega)>0$, which by integration
yields the required contradiction, concluding the proof.
\end{proof}
\section{More general results}\label{generalizations}
This section is devoted to generalize our results to problem
(\ref{eq:P}), where $Q$ and $f$ satisfy suitable assumptions which
generalize the specific behaviour of $\left|u_{t}\right|^{m-2}u_{t}$
and $\left|u\right|^{p-2}u$. Our assumptions on $Q$ are the
following ones.
\renewcommand{\labelenumi}{(Q\arabic{enumi})}
\begin{enumerate}
    \item $Q$ is a Carath\'eodory real function defined on $(0,\Theta)\times\Gamma_{1}\times\mathbb{R}$ for some $\Theta>0$,
     $Q(t,x,0)=0$ for almost all $(t,x)\in(0,\Theta)\times\Gamma_{1}$, and there exist an exponent $m > 1$ and positive constants
     $c_{1},c_{2},c_{3}$ and $c_{4}$, possibly dependent on $\Theta$, such that
\begin{alignat*}{2}
c_{1}\left|v\right|^{m-1}\leq &\left|Q(t,x,v)\right|\leq
c_{2}\left|v\right|^{m-1} &&\mbox{   when } \left|v\right|\geq 1
\qquad\text{and} \\
c_{3}\left|v\right|^{m-1}\leq &\left|Q(t,x,v)\right|\leq c_{4}
&&\mbox{ when } \left|v\right|\leq 1
\end{alignat*}
for almost all $(t,x)\in(0,\Theta)\times\Gamma_{1}$ and all
$v\in\mathbb{R}$.

\vspace{0.2cm}
  \item  The function $Q(t,x,\cdot)$ is increasing for almost all $(t,x)\in(0,\Theta)\times\Gamma_{1}$.

\end{enumerate}

\begin{rem} \label{remark6} When $Q=Q(v)$ assumptions (Q1)--(Q2) reduce (independently on $\Theta$)
 to assume that
 $Q\in C(\R)$ is increasing and such that
$$ Q(0)=0,\qquad\varliminf_{v\rightarrow 0}\frac{\left|Q(v)\right|}{\left|v\right|^{m-1}} > 0,
\qquad
0<\varliminf_{|v|\rightarrow\infty}\frac{\left|Q(v)\right|}{\left|v\right|^{m-1}}
\le\varlimsup_{|v|\rightarrow\infty}\frac{\left|Q(v)\right|}{\left|v\right|^{m-1}}
<\infty ,$$ as for example $Q=Q_0(v)=a |v|^{\mu-2}v+b
|v|^{m-2}v,\qquad a\ge 0,\quad b>0,\quad 1<\mu\le m$. Moreover
(Q1--2) are also satisfied for any  $\Theta>0$ by
$Q=Q_1(t,v)=d(t)Q_0(v)$, where $d\in
L^\infty_{\text{loc}}([0,\infty))$, $d>0$, $1/d\in
L^\infty_{\text{loc}}([0,\infty))$.
\end{rem}

\begin{rem} Let us note, for a future use, that (Q1)--(Q2) yield the
existence of positive constants $c_5$ and $c_6$ (possibly dependent
on $\Theta$) such that
\begin{gather}
\left|Q(t,x,v)\right|\leq c_5(1+\left|v\right|^{m-1})
\label{eq:OQ1}\\ \intertext{and}
Q(t,x,v)v\geq c_6\left|v\right|^{m}
\label{eq:OQ2}
\end{gather}
for almost all $(t,x)\in(0,\Theta)\times\Gamma_{1}$ and all
$v\in\mathbb{R}$.
\end{rem}

\subsection{Forced heat equation}

We first present our generalization of  Theorem \ref{eq:T1} to the
problem
\begin{equation}
\mbox{ $\left\{\begin{array}{llll}
$$u_{t}-\Delta u=g(t,x) & \mbox{in } (0,T)\times\Omega,$$\\
$$u=0 & \mbox{on } \left[0,T\right)\times\Gamma_{0},$$\\
$$\frac{\partial u}{\partial\nu} = -Q(t,x,u_{t}) & \mbox{on } \left[0,T\right)\times\Gamma_{1},$$\\
$$u(0,x)=u_{0}(x) & \mbox{in } \Omega$$
\end{array}
\right.$}\label{eq:P8}
\end{equation}
where $g$ is a given term acting on $\Omega$ and $T > 0$ is fixed.
\begin{definition} Let $u_{0}\in H^{1}_{\Gamma_{0}}(\Omega)$ and $g\in L^{2}((0,T)\times\Omega)$.
We say that $u$  is a weak solution of $(\ref{eq:P8})$ in
$[0,T]\times\Omega$ if (a--d) of Definition~\ref{eq:def1} hold, with
the distribution identity \eqref{eq:def2} being replaced by
\begin{equation}
\int_{\Omega}u_{t}(t)\phi + \int_{\Omega}\nabla u(t)\nabla \phi +
\int_{\Gamma_{1}}Q(t,\cdot,u_t(t))\,u_{t}(t)\phi=\int_{\Omega}g(t)\phi,
\label{distrforcedQ}
\end{equation}
which makes sense due to \eqref{eq:OQ1}.
\end{definition}
\begin{thm} \label{eq:T6} Suppose that (Q1) and  (Q2) hold with $\Theta=T$ and that
    \mbox{$g\in L^{2}((0,T)\times\Omega)$}.
Then, given any initial datum $u_{0}\in H^{1}_{\Gamma_{0}}(\Omega)$,
there is a unique weak solution $u$ of $(\ref{eq:P8})$ in
$[0,T]\times\Omega$. Moreover $(\ref{eq:4.1})$ and $(\ref{eq:4.2})$
hold and $u$ satisfies the energy identity
\begin{equation*}
\frac{1}{2}\left\|\nabla u\right\|^{2}_{2}\Big
|^{t}_{s}+\int^{t}_{s}\left\|u_{t}\right\|^{2}_{2}+\int^{t}_{s}\int_{\Gamma_{1}}Q(\cdot,\cdot,u_{t})u_{t}=\int^{t}_{s}\int_{\Omega}gu_{t}
\end{equation*}
for $0\leq s\leq t\leq T$. Finally, given any couple of initial data
$u_{01}, u_{02}\in H^{1}_{\Gamma_{0}}(\Omega)$ and any couple of
forcing terms $g_1,g_2\in L^{2}((0,T)\times\Omega)$, denoting by
$u^1$ and $u^2$ the solutions of \eqref{eq:P8} respectively
corresponding to $u_{01}$,  $g_1$ and to $u_{02}$, $g_2$, the
estimate \eqref{hadamardaux} holds true.
\end{thm}
\begin{proof}[\bf Sketch of the proof of Theorem \ref{eq:T6}]
Using \eqref{eq:OQ1}, \eqref{eq:OQ2} and (Q2) we can  repeat almost
verbatim the proof of Theorem \ref{eq:T1} by replacing everywhere
$|u_t^k|^{m-2}u_t^k$ with $Q(t,x,u_t^k)$, so  starting from the
problem
\begin{equation}
\mbox{ $\left\{\begin{array}{ll}
$$(u^{k}_{t},w_{j})+(\nabla u^{k}_{t},\nabla w_{j})+\int_{\Gamma_{1}}Q(\cdot,\cdot,u^{k}_{t})w_{j}=\int_{\Omega}gw_{j}, & \mbox{    } j=1,\ldots,k,$$\\
$$u^{k}(0)=u_{0k}. $$
\end{array}
\right.$} \label{eq:P9}
\end{equation}
The definition \eqref{sdai} is now replaced by $ G_{k}(t,y)=
y+\int_{\Gamma_{1}}Q(t,\cdot,B_{k}(x)\cdot y)B_{k}(x)dx$, $t\in
(0,T)$, $y\in\mathbb{R}^{k}$, so in the generalization of
problem~\eqref{eq:1.2} now $G_k$ explicitly depends on $t$. By using
assumption (Q2) the arguments of \cite[Proof of Theorem
1.5]{globalheat} continue to work in this more general situation for
any fixed $t\in (0,T)$, while all the other estimates keep
unchanged. The energy identity \eqref{eq:1.12} continues to hold
provided the term $\|u_t^k\|_{m,\Gamma_1}$ is replaced by the term
$\int_{\Gamma_1} Q(t,x,u^k_t)u_t^k$. By using \eqref{eq:OQ2} and
(Q2) we still get \eqref{eq:1.5} with $|u_t^k|^{m-2}u_t^k$ being
replaced by  $Q(t,x,u_t^k)$ in the forth line, where now $C'$
depends also on $c_1$ -- $c_4$.  Finally, to apply the monotonicity
method we use (Q2), which is also used in the proof of estimate
\eqref{hadamardaux}.
\end{proof}
\subsection{Local well--posedness}
We generalize Theorem \ref{Thloc} to problem \eqref{eq:P} under the
following  assumption on $f$:
\renewcommand{\labelenumi}{(F\arabic{enumi})}
\begin{enumerate}
    \item   $f$ is a Carath\'eodory real function defined on
$\Omega\times\mathbb{R}$, $f(x,0)=0$ for almost all $x\in\Omega$ and
there is an exponent $p\geq 2$ and a positive constant $c_7$ such
that for almost all $x\in\Omega$ and all $u_{1},u_{2}\in\mathbb{R}$
    \[\left|f(x,u_{1})-f(x,u_{2})\right|\leq c_7\left|u_{1}-u_{2}\right|(1+\left|u_{1}\right|^{p-2}+\left|u_{2}\right|^{p-2}).
\]
\end{enumerate}
An explicit example of a function $f$ which satisfies (F1) (use
\eqref{**ester}) is given by
\begin{equation}\label{f0}
f=f_0(x,u)=a(x) |u|^{q-2}u+b(x)|u|^{p-2}u,\qquad 2\le q\le p,\quad
a,b\in L^\infty(\Omega). \end{equation}
 When $f$ is independent on $x$
assumption (F1) can be equivalently written as follows:
 \begin{equation} \label{f122}
 f\in
W^{1,\infty}_{loc}(\mathbb{R}),\mbox{    }f(0)=0,\mbox{
}\left|f'(u)\right|=O\left(\left|u\right|^{p-2}\right)\mbox{  as
}\left|u\right|\rightarrow\infty.
\end{equation}

A further example of a non--algebraic nonlinearity satisfying
\eqref{f122} is given by $f=\pm f_1$, where $f_1(u)= |u|^{p-2}u$,
$p\ge 2$, when $|u|\ge 1$ while $f_1(u)=u$ when $|u|\le 1$.

\begin{rem} We note that an immediate consequence of (F1) is the
existence of a positive constant $c_8$ such that
\begin{equation}
\left|f(x,u)\right|\leq c_8(\left|u\right|+\left|u\right|^{p-1})
\label{eq:OF1}
\end{equation}
for almost all $x\in\Omega$ and all $u\in\mathbb{R}$.
\end{rem}

\begin{definition} Let $u_{0}\in H^{1}_{\Gamma_{0}}(\Omega)$ and suppose that (Q1-2), (F1) and
 assumption \eqref{eq:pr} hold.
We say that $u$ is a weak solution of \eqref{eq:P} in
$[0,T]\times\Omega$ if (a--d) of Definition~\ref{eq:def1} hold, with
the distribution identity \eqref{eq:def2} being replaced by
\begin{equation}
\int_{\Omega}u_{t}(t)\phi + \int_{\Omega}\nabla u(t)\nabla \phi +
\int_{\Gamma_{1}}Q(t,x,u_{t}(t))\phi=\int_{\Omega}f(x,u(t))\phi.
\label{eq:DEF.1}
\end{equation}
Moreover we say that $u$ is a weak solution of problem \eqref{eq:P}
in $[0,T)\times\Omega$ if it is  a weak solution of \eqref{eq:P} in
$[0,T']\times\Omega$ for all $T'\in (0,T)$.
\end{definition}

Theorem~\ref{Thloc} is generalized as follows.

\begin{thm} \label{eq:T7} Suppose that (Q1), (Q2) and (F1) hold together with
\eqref{eq:pr}. Then given any initial datum $u_{0}\in
H^{1}_{\Gamma_{0}}(\Omega)$ there is
$T^{*}=T^{*}(\left\|u_{0}\right\|_{H^{1}_{\Gamma_{0}}},m,p,\Omega,\Gamma_{1},c_1,c_2,c_3,c_4)$
in $(0,\min\{1,\Theta\}]$, decreasing in the first variable, such
that problem \eqref{eq:P} has a unique weak solution in
$[0,T^*]\times\Omega$. Moreover \eqref{eq:4.12}, \eqref{eq:4.13} and
\eqref{eq:gengis} hold, together with the energy identity
\begin{equation}
\frac{1}{2}\left\|\nabla u\right\|^{2}_{2}\Big
|^{t}_{s}+\int^{t}_{s}\left\|u_{t}\right\|^{2}_{2}+\int^{t}_{s}\int_{\Gamma_{1}}Q(\cdot,\cdot,u_{t})u_{t}=\int^{t}_{s}\int_{\Omega}f(\cdot,u)u_{t}
\label{eq:4.25}
\end{equation}
for $0\leq s\leq t\leq T^{*}$.

\end{thm}
\begin{proof}[\bf Sketch of the proof]
 We repeat the proof of Theorem \ref{Thloc}. We take $T\leq\Theta$
 and  $u\in X_T$. We note that, by \eqref{eq:OF1} and \eqref{eq:pr},
 we have $f(\cdot,u)\in
L^\infty(0,T;L^2(\Omega))$. So by
  Theorem \ref{eq:T6} there is  a unique solution $v$ of the
  problem
\begin{equation}
\mbox{ $\left\{\begin{array}{llll}
$$v_{t}-\Delta v=f(x,u) & \mbox{in } (0,T^{*})\times\Omega,$$\\

$$v=0 & \mbox{on } \left[0,T^{*}\right)\times\Gamma_{0},$$\\

$$\frac{\partial v}{\partial\nu} = -Q(t,x,v_{t}) & \mbox{on } \left[0,T^{*}\right)\times\Gamma_{1},$$\\

$$v(0,x)=u_{0}(x) & \mbox{in } \Omega.$$
\end{array}
\right.$}\label{eq:P10}
\end{equation}
We set $\Phi:X_T\to X_T$ by $\Phi(u)=v$. By using the same arguments
in the proof of Theorem \ref{Thloc} together with assumptions (Q2)
and (F1)  we get for any $u\in B_R$, the estimate
\begin{equation}\label{Coco521}
\frac 12\left\|\nabla v(t)\right\|^2_2 + \frac 12 \int^t_0
\left\|v_{t}\right\|^2_2 \leq  \frac 12\left\|\nabla
u_{0}\right\|^2_2+ K_2 (R^2+R^{2(p-1)}) T\end{equation}
 which
generalizes \eqref{Coco52} to this more general situation, where now
the constants $K_i$ depends also on $c_7$. Then we proceed  as in
the quoted proof  with $(R^{2}+R^{2(p-1)})$ replacing $R^{2(p-1)}$.
Consequently we get that $\Phi(B_R)\subset B_R$ provided that
\begin{equation}\label{Tnew}
R=4 R_0, \qquad\text{and}\quad T \leq \min\left\{1,\Theta,
K_3(16+16^{p-1}{R_0}^{2(p-2)})^{-1}\right\},
\end{equation}
generalizing \eqref{eq:T}. In order to show that, for suitable $T$,
$\Phi$ is a contraction in $B_R$ we proceed exactly as in the quoted
proof by taking $u,\bar{u}\in B_{R}$, $v=\Phi(u)$,
$\bar{v}=\Phi(\bar{u})$, $w=v-\bar{v}$. Clearly, $w$ is a weak
solution of the problem
\begin{equation}
\mbox{ $\left\{\begin{array}{llll}
$$w_{t}-\Delta w=f(x,u)-f(x,\bar{u}) & \mbox{in } (0,T)\times\Omega,$$\\
$$w=0 & \mbox{on } \left[0,T\right)\times\Gamma_{0},$$\\
$$\frac{\partial w}{\partial\nu} = -Q(t,x,v_t)+Q(t,x,\bar{v}_{t}
& \mbox{on } \left[0,T\right)\times\Gamma_{1},$$\\
$$w(0,x)=0 & \mbox{in } \Omega.$$
\end{array}
\right.$}\label{eq:P6BBIS}
\end{equation}
generalizing \eqref{eq:P6}. Since by \eqref{eq:OF1}  we have
$f(\cdot,u), f(\cdot,\bar{u})\in L^\infty(0,T;L^2(\Omega))$ and by
\eqref{eq:OQ1} we have $Q(\cdot,\cdot,v_t),
Q(\cdot,\cdot,\bar{v}_t)\in L^{m'}((0,T)\times\Gamma_1)$  we can
apply  Lemma~\ref{eq:l1} to get
\begin{multline*}
\frac{1}{2}\left\|\nabla w\right(t)\|^{2}_{2}+
\int^{t}_{0}\left\|w_{t}\right\|^{2}_{2}+
\int^{t}_{0}\int_{\Gamma_{1}}\left[Q(\cdot,\cdot,v_{t})-Q(\cdot,\cdot,\bar{v}_{t})\right]w_{t}\\
=\int^{t}_{0}\int_{\Omega}\left[f(\cdot,u)-f(\cdot,\bar{u})\right]w_{t}.
\end{multline*}
Using (F1) and (Q2) we generalize the estimate \eqref{spinoza} to
the following one
\begin{equation}
\frac{1}{2}\left\|\nabla w(t)\right\|^{2}_{2}+
\int^{t}_{0}\left\|w_{t}\right\|^{2}_{2} \leq
c_{7}\int^{T}_{0}\int_{\Omega}\left|u-\bar{u}\right|(1+\left|u\right|^{p-2}+\left|\bar{u}\right|^{p-2})\left|w_{t}\right|.
\label{eq:4.20}
\end{equation}
Consequently exactly the same arguments used in the quoted proof
allow to prove the estimate
\begin{equation}
\left\|w(t)\right\|^{2}_{H^{1}_{\Gamma_{0}}}  \leq
K^{2}_{13}(1+R^{(p-2)})^2
T\left\|u-\bar{u}\right\|^{2}_{L^{\infty}(0,T;H^{1}_{\Gamma_{0}}(\Omega))}.
\label{3cruncnew}\end{equation} replacing \eqref{3crunc}, so by
\eqref{Tnew}$_1$, $\Phi$ is a contraction provided $T
<K_{13}^{-2}(1+4^{p-2} R_0^{p-2})^{-2}$.  We can the finally fix
$T^*$ and complete the proof.
\end{proof}
The following result is nothing but the generalization of
Theorem~\ref{Thalt}.
\begin{thm} \label{eq:T8}
Suppose that (Q1--2) hold for all $\Theta>0$, together with (F1) and
\eqref{eq:pr}. Then the assertions of Theorem~\ref{Thalt} hold when
problem \eqref{Pmod} is replaced by  \eqref{eq:P}.
\end{thm}
\begin{proof}[\bf Sketch of the proof] We describe the adaptations
needed to cover this more general situation with respect to the
arguments used in the proof of Theorem~\ref{Thalt}. The existence of
a unique weak maximal solution of \eqref{eq:P} follows exactly in
the same way. When proving the alternative (i--ii), since the
equation is  not autonomous (as the term $Q$ is explicitely
time--dependent) a more detailed explanation is needed. Let us
suppose by contradiction that \eqref{alternative} holds, so there is
a sequence $T_n\to T_{\text{max}}^-<\infty$ such that
$\|u(T_n)\|_{H^1_{\Gamma_0}}$ is bounded. Since $Q$ satisfies
assumptions (Q1--2) for all positive $\Theta$, we can choose
$\Theta=T_{\text{max}}+1$. We set for any $n\in\N$ the time--shifted
nonlinear term $Q_n(t,x,v)=Q(t+T_n,x,v)$, which satisfies
assumptions (Q1--2) with $\Theta=\Theta_n:=T_{\text{max}}-T_n+1\ge
1$, so that  $Q_n$ satisfies the same assumptions for $\Theta=1$ for
all $n\in\N$. It follows that the existence time $T^*$ assured by
Theorem~\ref{eq:T7} is independent on $n$, so problem \eqref{eq:P}
with initial time $T_n$ and initial datum $u(T_n)$ has a unique weak
solution in $[T_n,T_n+T^*]\times\Omega$, which leads to the desired
contradiction.

When proving the continuous dependence of the solution $u$ on the
initial datum we get the energy identity
\begin{multline*}
\frac{1}{2}\left\|\nabla w^n\right\|^{2}_{2} \Big
|^{t}_{0}+\int^{t}_{0}\left\|w^n_t\right\|^{2}_{2}
+\int^{t}_{0}\int_{\Gamma_{1}}\left[Q(\cdot,\cdot,u^n_t)-Q(\cdot,\cdot,u_t)\right]w^n_t \\
=\int^{t}_{0}\int_{\Omega}\left[f(\cdot,u^n)-f(\cdot,u)\right]w^n_t
\end{multline*}
generalizing \eqref{eq:rhs}. Using assumptions (Q2), (F1) and
\eqref{eq:pr} we then get  the estimate \eqref{fascio} again, so we
can conclude the proof exactly as in Theorem \ref{Thalt}.
\end{proof}

\subsection{Global existence versus blow--up}
In order to generalize Theorems~\ref{Thglo} and \ref{Thblow} to
problem (\ref{eq:P}) we first generalize Lemma~\ref{auxiliary}. We
introduce the notation $$F(x,u)=\int^{u}_{0}f(x,s)ds.$$
\begin{lem}\label{auxiliarygener}
Under the assumptions of Theorem~\ref{eq:T7}, let $u$ be a weak
solution of problem \eqref{eq:P} in $[0,T]\times\Omega$. Then
\begin{equation}\label{auxiliaryformulagen}
\frac d{dt}F(\cdot,u(t))=\int_\Omega f(\cdot,
u(t))\,u_t(t)\qquad\text{for almost all $t\in (0,T)$.}
\end{equation}
\end{lem}
\begin{proof}
We first note that an immediate consequence of \eqref{eq:OF1} is
that $|F(x,u)|\leq c_9 (1+|u|^p)$ for a positive constant $c_9$.
Hence $\int_\Omega F(\cdot, u)\in L^\infty(0,T)\subset L^2(0,T)$.
Consequently exactly the same arguments used in the proof of
Lemma~\ref{auxiliary} apply to this more general case.
\end{proof}

To extend in a suitable way the definition of the stable and
unstable sets we need to introduce a second structural assumption on
the nonlinearity $f$.
\renewcommand{\labelenumi}{(F2)}
\begin{enumerate}
\item There is $c_{10} \geq 0$ such that
$F(x,u)\leq \dfrac{c_{10}}{p}\left|u\right|^{p}$ for almost all
$x\in\Omega$ and all $u\in\mathbb{R}$.
\end{enumerate}

We remark that the model nonlinearity $f_0$ defined in \eqref{f0}
satisfies (F2) if and only if $a\le 0$.

When $\sigma(\Gamma_{0})>0$, $p>2$ and (F2) holds we set
\begin{equation}\label{D1gen}
D_{1}:=\sup_{u\in H^{1}_{\Gamma_{0}}(\Omega), u\neq
0}\frac{\int_{\Omega}F(\cdot,u)}{\left\|\nabla
u\right\|^{p}_{2}}\leq \frac{c_{10}}{p}B^{p}_{1}.
\end{equation}
When $D_1>0$ we also set
\begin{equation}\label{DefE1}
    \lambda_{1}=(pD_{1})^{-1/(p-2)},\quad
    E_{1}=\left(\frac{1}{2}-\frac{1}{p}\right)\lambda^{2}_{1},
\end{equation}
while $\lambda_1=E_1=+\infty$ when $D_1\le 0$. Moreover we denote
\begin{align}\label{defJgen}
&J(u)=\frac 12 \|\nabla u_0\|_2^2-\int_\Omega
F(\cdot,u)\qquad\text{for
any $u\in H^1_{\Gamma_0}(\Omega)$,} \\
&W_{s}=\left\{u_{0}\in H^{1}_{\Gamma_{0}}(\Omega): \|\nabla
u_0\|<\lambda_1 \mbox{ and } J(u_{0})<E_1\right\}, \label{eq:Wsgen}\\
&W_{u}=\left\{u_{0}\in H^{1}_{\Gamma_{0}}(\Omega): \|\nabla
u_0\|>\lambda_1 \mbox{ and } J(u_{0})<E_1\right\}. \label{eq:Wugen}
\end{align}
Clearly due to Lemma~\ref{eq:Ld} when $f=|u|^{p-2}u$ definitions
\eqref{eq:Wsgen} and \eqref{eq:Wugen} concide with \eqref{eq:Ws} and
\eqref{eq:Wu}, even if they are inspired from \eqref{W1} and
\eqref{W2}.

We now generalize the potential--well argument contained in
Lemma~\ref{rem}.
\begin{lem} \label{gensu}
Suppose that (Q1--2) hold for all $\Theta>0$, together with (F1--2) and
\eqref{eq:pr}. Suppose moreover that $\sigma(\Gamma_{0})>0$ and
$p>2$. Then the conclusion of Lemma~\ref{rem} continue to hold.
\end{lem}

\begin{proof}
By Lemma~\ref{auxiliarygener} the energy identity \eqref{eq:4.25}
can be written as
\begin{equation}\label{enid7gen}
J(u(\tau))\Big |_s^t=-\int_s^t \int_{\Gamma_1} Q(\tau,\cdot,
u_t(\tau))\,u_t(\tau)\,d\tau-\int_s^t\|u_t\|_2^2\,d\tau.
\end{equation}
By \eqref{enid7gen} and (Q2) the energy function $E(t):=J(u(t))$ is
decreasing in $[0,T_{\text{max}})$. Hence \eqref{eq:et} continue to
hold. By \eqref{D1gen} we have $J(u(t))\ge \widetilde{g}(\|\nabla
u(t)\|_2)$, where $\widetilde{g}= \frac {\lambda^2}2-D_1 \lambda^p$
if $D_1>0$, while $\widetilde{g}= \frac {\lambda^2}2$ if $D_1\le 0$.
Then, when $D_1>0$, the same arguments used in the proof of
Lemma~\ref{rem} apply, while there is nothing to prove when $D_1\le
0$.
\end{proof}

We can now state the generalization of Theorem~\ref{Thglo}.

\begin{thm} Under  the assumptions of Lemma~\ref{gensu}
if $u_{0}\in W_{s}$ then $T_{max}=\infty$ and $u(t)\in W_{s}$ for
all $t\geq 0$.
\end{thm}

\begin{proof}
When $D_1>0$ we can exactly repeat the proof of Theorem~\ref{Thglo}
by using Lemma~\ref{gensu}. When  $D_1\le 0$ the same argument
applies since in this case we have $J(u)\ge \frac 12 \|\nabla
u\|_2^2$ so $W_s$ is bounded.
\end{proof}

In order to generalize Theorem~\ref{Thblow} we need to strengthen
assumption (Q1--2) to the following ones.
\renewcommand{\labelenumi}{(Q\arabic{enumi}$'$)}
\begin{enumerate}
\item $Q$ is a Carath\'eodory real function defined on
$(0,\infty)\times\Gamma_{1}\times\mathbb{R}$,
     $Q(t,x,0)=0$ for almost all $(t,x)\in(0,\infty)\times\Gamma_{1}$,
     and there exists exponents $1<\mu\le m$, a positive function $d$ such that $d, 1/d \in
L^\infty_{\text{loc}}([0,\infty))$,
      and positive constants $c'_1,c'_2,c'_3$ and $c'_4$  such that
\begin{alignat*}{2}
c'_{1}d(t)\left|v\right|^{m-1}\leq &\left|Q(t,x,v)\right|\leq
c'_{2}d(t)\left|v\right|^{m-1} &&\mbox{   when } \left|v\right|\geq
1,
\quad\text{and}\\
c'_{3}d(t)\left|v\right|^{\mu-1}\leq &\left|Q(t,x,v)\right|\leq
c'_{4}d(t) \left|v\right|^{\mu-1}  &&\mbox{ when }
\left|v\right|\leq 1
\end{alignat*}
for almost all $(t,x)\in(0,\infty)\times\Gamma_{1}$ and all
$v\in\mathbb{R}$.
\item The function $Q(t,x,\cdot)$ is increasing for almost all $(t,x)\in(0,\infty)\times\Gamma_{1}$.

\end{enumerate}

\begin{rem}
We remark that the nonlinearities $Q_0$ and $Q_1$ defined in
Remark~\ref{remark6} satisfy as well assumption (Q1$'$--2$'$).
Moreover when $Q=Q(v)$ these assumptions  reduce to assume that
$Q\in C(\R)$ is increasing and
$$ 0<\varliminf_{v\rightarrow
0}\frac{\left|Q(v)\right|}{\left|v\right|^{\mu-1}}\le
\varlimsup_{v\rightarrow
0}\frac{\left|Q(v)\right|}{\left|v\right|^{\mu-1}}<\infty , \qquad
0<\varliminf_{|v|\rightarrow\infty}\frac{\left|Q(v)\right|}{\left|v\right|^{m-1}}
\le\varlimsup_{|v|\rightarrow\infty}\frac{\left|Q(v)\right|}{\left|v\right|^{m-1}}
<\infty.$$ We also note, for future use,  some further consequences
of (Q1$'$) and (Q2$'$).  Since $Q(0)=0$, by (Q2) we have
$Q(t,x,v)v\ge 0$, so $Q(t,x,v)v=|Q(t,x,v)||v|$. Hence, when $|v|\ge
1$ we have
\begin{equation}\label{P1}
|Q(t,x,v)|\le c'_5 d^{1/m}(t)[Q(t,x,v)v]^{1/m'}
\end{equation}
while when $|v|\le 1$
\begin{equation}
|Q(t,x,v)|\le c'_6 d^{1/\mu}(t)[Q(t,x,v)v]^{1/\mu'}\label{P2}
\end{equation}
for almost all $(t,x)\in (0,\infty)\times \Gamma_1$, where  $c'_5$
and $c'_6$ are positive constants.
\end{rem}

In order to state our blow--up result for problem \eqref{eq:P} we
need a further specific structural assumption on $f$.
\renewcommand{\labelenumi}{(F3)}
\begin{enumerate}
  \item There is $\varepsilon_{0} > 0$ such that for all $\varepsilon\in\left(0,\varepsilon_{0}\right]$
  there exists $c_{11}=c_{11}(\varepsilon) > 0$ such that
\begin{equation*}
f(x,u)u-(p-\varepsilon)F(x,u)\geq c_{11}\left|u\right|^{p}
\end{equation*}
for almost all $x\in\Omega$ and all $u\in\mathbb{R}$.
\end{enumerate}

Clearly the model nonlinearity $f_0$ defined in \eqref{f0} satisfies
(F2--3) if and only if $a\le 0$. We can finally state
\begin{thm} \label{eq:T9}
Suppose that (Q1$'$--Q2$'$), (F1--3), \eqref{eq:pr} and
\eqref{eq:w4} hold. Moreover suppose that $\sigma(\Gamma_0)>0$,
$p>2$,
\begin{equation}\label{fivepointsstar}
\int^\infty \dfrac {dt}{d^{1/(m-1)}+d^{1/(\mu-1)}}=\infty
\end{equation}
and $u_0\in W_u$. Then the conclusions of Theorem~\ref{Thblow} hold.
\end{thm}

\begin{rem} Assumption \eqref{fivepointsstar} needs some
comment, as it express the possible time--behavior of $Q$. When
$d(t)=(1+t)^\beta$, $\beta\in\R$, it reduces to $\beta\le \mu-1$,
and in particular when $\mu=m$ in assumption (Q1$'$) (what happens
for example when $Q(v)=d(t)|v|^{m-2}v$), it reduces to $\beta\le
m-1$, which is a well--known optimal assumption to prevent
over--damping for time dependent damping terms in ordinary
differential systems.
\end{rem}

\begin{proof}
As in the proof of Theorem \ref{Thblow} we prove, by contradiction,
that there are no solutions in the whole $(0,\infty)\times\Omega$.
We fix $E_{2}\in(J(u_{0}),E_{1})$ and set $H$ by \eqref{eq:Ht1}. By
using Lemma~\ref{gensu} and \eqref{D1gen} we get a slightly
generalized version of \eqref{eq:Ht}, that is
\begin{equation}H(t)\le \frac
{c_{10}}p\|u(t)\|_p^p.\label{doppiascure} \end{equation} By (Q2$'$)
formula \eqref{Hprime} is now generalized to
\begin{equation}H'(t)=\int_{\Gamma_1}Q(t,\cdot,u_t)u_t+\|u_t(t)\|_2^2\geq
\int_{\Gamma_1}Q(t,\cdot,u_t)u_t\ge 0 \label{Hprimegen}
\end{equation}
so that \eqref{eq:w6} holds true. By \eqref{eq:OQ1} we can again
take $\phi=u_t$ in the distribution identity \eqref{eq:DEF.1} so
getting the following generalized version of \eqref{eq:w3}
\begin{equation}
\int_{\Omega}f(\cdot,u)u-\left\|\nabla
u\right\|^{2}_{2}=\int_{\Gamma_{1}}Q(\cdot,\cdot,u_{t})u+(u_{t},u)
\label{eq:4.29}
\end{equation}
The estimate \eqref{89} of the second term in the right hand side of
\eqref{eq:4.29} keeps unchanged, while the estimate the first term
in it needs a more detailed explanation. We use \eqref{P1},
\eqref{P2} and H\"{o}lder inequality twice to get
\begin{equation*}
\begin{aligned}
I_1:=&\left|\int_{\Gamma_{1}}Q(\cdot,\cdot,u_{t})u\right|\\
\le & c'_5d^{1/m}\int_{\{x\in \Gamma_1: |u_t|\ge
1\}}\negqquad\negqquad[Q(t,\cdot,u_t)u_t]^{1/m'}|u| +
c'_6d^{1/\mu}\int_{\{x\in
\Gamma_1: |u_t|\le 1\}}\negqquad\negqquad[Q(t,\cdot,u_t)u_t]^{1/\mu'}|u|\\
 \le & (c'_5+c'_6)\left[d^{1/m}\left(\int_{\Gamma_1}
Q(t,\cdot,u_t)u_t]\right)^{1/m'}+d^{1/\mu}\left(\int_{\Gamma_1}
Q(t,\cdot,u_t)u_t]\right)^{1/\mu'}\right]\|u\|_{m,\Gamma_1}\label{Qestimate1}
\end{aligned}
\end{equation*}
which generalizes \eqref{eq:w1}. Now we estimate
$\|u\|_{m,\Gamma_1}$ in previous formula by using \eqref{eq:w2}. In
this way we obtain
$$I_1\le C_3\left[d^{1/m}\left(\int_{\Gamma_1}
Q(t,\cdot,u_t)u_t]\right)^{1/m'}+d^{1/\mu}\left(\int_{\Gamma_1}
Q(t,\cdot,u_t)u_t]\right)^{1/\mu'}\right]\|u\|_2^{1-s}\|\nabla
u\|_2^s$$ for exponents $s$ satisfying \eqref{85} (generalizing
\eqref{458}). The same arguments used in the proof of
Theorem~\ref{Thblow} then give, for any $\delta>0$,
$$
\begin{aligned}
I_1\le & C_3\left[C_4(\delta)d^{1/(m-1)}\int_{\Gamma_1}
Q(t,\cdot,u_t)u_t+\delta \|u\|_2^2+\delta
\|u\|_p^p\right]\|u\|^{1-s-p(1/m-s/2)}_{p}\\
+& C_3\left[C_4(\delta)d^{1/(\mu-1)}\int_{\Gamma_1}
Q(t,\cdot,u_t)u_t+\delta \|u\|_2^2+\delta
\|u\|_p^p\right]\|u\|^{1-s-p(1/\mu-s/2)}_{p}
\end{aligned}
$$
provided also \eqref{86bis} (and consequently $s<2/\mu$ as well)
holds. By \eqref{eq:w6} and \eqref{doppiascure} we have
$\|u\|_p^p\geq \left(\frac p{c_{10}}H(0)\right)^{1/p}$, so from
previous formula we derive, as $\mu\le m$,
$$I_1\le C'_3 \left\{C_4(\delta) \left[d^{\frac 1{m-1}}+d^{\frac 1{\mu-1}}\right]\int_{\Gamma_1}
\negquad Q(t,\cdot,u_t)u_t +\delta \|u\|_2^2+\delta
\|u\|_p^p\right\}\|u\|^{-p\bar{\alpha_s}}_{p},
$$
where $C'_3=C'_3(p,m,s,\Omega,H(0))>0$ and
$\bar{\alpha}_{s}=-\left[1-s-p(1/m-s/2)\right]/p
> 0$, generalizing \eqref{88}. By plugging the last estimate and
\eqref{89} in \eqref{eq:4.29} and using \eqref{Hprimegen} we get
$$\int_{\Omega}f(\cdot,u)u-\left\|\nabla u\right\|^{2}_{2} \leq
C_{8} (\delta)\left[d^{\frac 1{m-1}}+d^{\frac 1{\mu-1}}\right]H'(t)
\|u\|^{-p\bar{\beta}_{s}}_{p}+C_{9}\delta\left[\left\|\nabla
u\right\|^{2}_{2}+ \left\|u\right\|^{p}_{p}\right]
$$
so generalizing \eqref{104}. Consequently, by \eqref{eq:Ht1} and
\eqref{defJgen} we have, for any $\varepsilon>0$,
\begin{multline*}
\int_{\Omega}f(\cdot,u)u+(p-\varepsilon)H(t)-(p-\varepsilon)E_2+\left[\frac
{p-\varepsilon}2-(1+C_9\delta)\right]\|\nabla
u\|_2^2\\-(p-\varepsilon)\int_\Omega F(\cdot, u)-C_9\delta \|u\|_p^p
\le C_{8} (\delta)\left[d^{\frac 1{m-1}}+d^{\frac
1{\mu-1}}\right]H'(t) \|u\|^{-p\bar{\beta}_{s}}_{p},
\end{multline*}
where $\bar{\beta}_s$ is given by \eqref{parametri}. Then, using
assumption (F3) for $\varepsilon\in (0,\varepsilon_0]$  we have
\begin{multline}\label{u158}
\left[c_{11}(\varepsilon)-C_9\delta\right]\|u\|_p^p+\left[\frac
{p-\varepsilon}2-(1+C_9\delta)\right]\|\nabla
u\|_2^2+(p-\varepsilon)H-(p-\varepsilon)E_2\\ \le C_{8}
(\delta)\left[d^{\frac 1{m-1}}+d^{\frac 1{\mu-1}}\right]H'(t)
\|u\|^{-p\bar{\beta}_{s}}_{p}.
\end{multline}
By Lemma~\ref{gensu} we have $\|\nabla u\|_2\ge \lambda_1$, so by
\eqref{DefE1} we get
\begin{equation}\label{u159}\left[\frac
{p-\varepsilon}2-(1+C_9\delta)\right]\|\nabla
u\|_2^2-(p-\varepsilon)E_2\ge \left(-C_9\delta -\frac\varepsilon
p\right)\lambda_1^2+ (p-\varepsilon)(E_1-E_2).
\end{equation}
 We fix
$\varepsilon=\eps_1$ small enough in order to have $\frac \eps p
\lambda_1^2<\frac {p-\eps}2 (E_1-E_2)$. After that we fix
$\delta=\delta_1$ such that $\frac{C_9\delta}2<\frac
{p-\eps_1}2(E_1-E_2)$ and $c_{11}(\eps_1)-C_9\delta>0$.
Consequently, from \eqref{u158} and \eqref{u159} we obtain
$$
(p-\varepsilon_1)H(t)\le C_{8} (\delta_1)\left[d^{\frac
1{m-1}}+d^{\frac 1{\mu-1}}\right]H'(t) \|u\|^{-p\bar{\beta}_{s}}_{p}
$$
generalizing \eqref{fest}. By \eqref{doppiascure} we finally obtain
$H'(t)\ge C_9 \left[d^{\frac 1{m-1}}+d^{\frac
1{\mu-1}}\right]H^{1+\bar{\beta}_s}(t)$ which generalizes
\eqref{fest2}. By integrating and using assumption
\eqref{fivepointsstar} we get the desired contradiction and conclude
the proof.
\end{proof}

\appendix
\section{A physical model} \label{eq:appA}
This section is devoted to describe a physical model which motivates
problem (\ref{eq:P}). Let $\Omega$ represent a solid body surrounded
by a fluid denoted by $A$, whit contact  $\Gamma_1$ and (possibly)
having an internal cavity with contact boundary $\Gamma_0$. We
suppose that a heat reaction-diffusion process occurs inside
$\Omega$ such that, if $u=u(t,x)$ represents the temperature at
point $x$ and time $t$, the quantity of heat produced by the
reaction is proportional to a superlinear power of the temperature,
i.e. to $u^{p-1}$ with $p
> 2$. Thus the process can be modelled by the heat equation with
source
\begin{equation}
u_{t}-\rho\Delta u=\left|u\right|^{p-2}u \mbox{   in
}(0,T)\times\Omega \label{eq:A1}
\end{equation}
where the thermal conductivity $\rho > 0$ is taken to be 1 for
simplicity. The surrounding fluid is supposed to be a perfect
conductor of heat, so the temperature in $A$ is spatially
homogeneous and can be described by a number $v=v(t)$ for any $t\geq
0$. In particular, there is no diffusion in the fluid. Such
assumption is realistic if the fluid is well stirred. Moreover, we
introduce a refrigerating process in the fluid with the help of
which one tries to control the reaction inside the solid $\Omega$.
We assume that the refrigerating system is controlled in such a way
that the heat absorbed from the fluid is proportional to a power of
the rate of change of the temperature, as
$\left|v'(t)\right|^{m-2}v'(t)$. Let $j=j(t,x)$ be the heat flux
from $\Omega$ to $A$. Then the rate of change of the temperature
$v'(t)$ is given by $ v'(t)=-\left|v'(t)\right|^{m-2}v(t)+
\int_{\Gamma_1}j(t,x)dS$. On the other hand, the heat flux $j(t,x)$
is given by the classical conductivity rule by
$j(t,x)=-\dfrac{\partial u}{\partial\nu}$, since $\rho=1$. Finally,
the thermal contact of the fluid at $\Gamma_{1}$ yields the
continuity condition $u(t,x)=v(t),\mbox{ }x\in\Gamma_{1},\mbox{
}t\geq 0$, while the temperature on $\Gamma_{0}$ is assumed to be
constant (for simplicity constantly vanishing), that is
\begin{equation}
u(t,x)=0,\mbox{   }x\in\Gamma_{0},\mbox{   }t\geq 0. \label{eq:A5}
\end{equation}
Combining (\ref{eq:A1})-(\ref{eq:A5}), we obtain (\ref{eq:P}) with
$f=\left|u\right|^{p-2}u$ and
 $Q=u_{t}+\left|u_{t}\right|^{m-2}u_{t}$. These nonlinear terms are included in theory developed in Section~\ref{generalizations}.
 In particular Theorem \ref{eq:T9} shows that the refrigerating system cannot avoid the internal explosion with this conditions.

\section{Global existence for problem \eqref{Pmod}  when $p=2$} \label{appendicep2}

This section is devoted to state and prove the  global existence
result for problem \eqref{Pmod} when $p=2$ mentioned in the
Introduction. For the sake of generality we actually shall prove a
more general version of it dealing with problem \eqref{eq:P}.

\begin{thm}\label{easytheorem}
Under the assumptions of Theorem~\ref{eq:T8} if $p=2$ then
$T_{\text{max}}=\infty$.
\end{thm}
\begin{proof}
We suppose by contradiction that $T_{\text{max}}<\infty$. By
\eqref{eq:4.25} together with assumptions (Q1--2) and \eqref{eq:OF1}
we have
$$\frac 12 \|\nabla u\|_2^2+\int_0^t \|u_t\|_2^2\le \frac 12 \|\nabla
u_0\|_2^2+2c_8 \int_0^t \int_\Omega |u||u_t|.$$ By H\"{o}lder and
weighted Young inequalities we consequently get
$$\frac 12 \|\nabla u\|_2^2+\int_0^t \|u_t\|_2^2\le \frac 12 \|\nabla
u_0\|_2^2+2c_8^2 \int_0^t \|u\|_2^2+\frac 12 \int_0^t\|u_t\|_2^2$$
and consequently \begin{equation}\label{appendiceC1}\|\nabla
u\|_2^2+\int_0^t \|u_t\|_2^2\le \|\nabla u_0\|_2^2+4c_8^2 \int_0^t
\|u\|_2^2.
\end{equation}
 Moreover, by integrating and using H\"{o}lder inequality
in time we have
\begin{equation}\label{appendiceC2}
\|u\|_2^2\leq \left(\|u_0\|_2+\int_0^t\|u_t\|_2\right)^2\le
2\|u_0\|_2^2+2T_{\text{max}}\int_0^t \|u_t\|_2^2.
\end{equation}
Combining \eqref{appendiceC1} and \eqref{appendiceC2} we get
$$\int_0^t \|u_t\|_2^2\le \|\nabla u_0\|_2^2+8T_{\text{max}}c_8^2
\left( \|u_0\|_2^2+\int_0^t ds \int_0^s
\|u_t(\tau)\|_2^2\,d\tau\right).$$ By Gronwall inequality we then
get that $\int_0^t \|u_t\|_2^2$ is bounded up to $T_{\text{max}}$.
By \eqref{appendiceC2} we consequently get that also $\|u\|_2$ is
bounded. Hence, by \eqref{appendiceC1} also $\|\nabla u\|_2$ is
bounded. So we contradict \eqref{bu} and conclude the proof.
\end{proof}

\section{Proof of Lemma \ref{eq:l1}}\label{Appendice A}

At first we denote
    $H=L^{2}(\Omega)$, $V=H^{1}_{\Gamma_{0}}(\Omega)$ and $W=L^{m}(\Gamma_{1})$.
Since $V$ is dense in $H$, using \cite[Theorem 2.1]{strauss} and
$(\ref{eq:b1})$, $(\ref{eq:b2})$, we obtain that
\begin{equation}
u\in C_{w}(\left[0,T\right];V). \label{eq:b9}
\end{equation}

The key point is to show that the energy identity holds. With this
aim and fixed $0\leq s\leq t\leq T$, we set $\theta_{0}$ to be the
characteristic function of the interval $\left[s,t\right]$. For
small $\delta > 0$, let $\theta(\tau)=\theta_{\delta}(\tau)$ be 1
for $\tau\in\left[s+\delta,t-\delta\right]$, zero for
$\tau\notin(s,t)$ and linear in the intervals
$\left[s,s+\delta\right]$ and $\left[t-\delta,t\right]$. Next let
$\eta_{\varepsilon}$ be a standard mollifying sequence, that is,
$\eta=\eta_{\varepsilon}\in C^{\infty}(\mathbb{R})$, supp
$\eta_{\varepsilon}\subset(-\varepsilon,\varepsilon)$,
$\int^{\infty}_{-\infty}\eta_{\varepsilon}=1$, $\eta_{\varepsilon}$
even and nonnegative, and
$\eta_{\varepsilon}=\varepsilon^{-1}\eta_{1}(\tau/\varepsilon)$. Let
$\ast$ denote time convolution. We approximate $u$, extended as zero
outside $\left[0,T\right]$, with $v=\eta\ast(\theta u)\in
C^{\infty}_{c}(\mathbb{R};V)$. Then
\begin{equation}
0=\int^{+\infty}_{-\infty}\frac{1}{2}\frac{d}{dt}\left\|\nabla
v\right\|^{2}_{2}=\int^{+\infty}_{-\infty}(\nabla v,\nabla v_{t}).
\label{eq:b3}
\end{equation}
Using standard convolution properties and the Leibnitz rule, we see
that
    $v_{t}=\eta\ast(\theta'u)+\eta\ast(\theta u_{t})$ in $H$,
so that $\eta\ast(\theta u_{t})\in C^{\infty}_{c}(\mathbb{R};V)$.
Then, by $(\ref{eq:b3})$,
\begin{equation}
0=\int^{+\infty}_{-\infty}(\eta\ast(\theta\nabla
u),\eta\ast(\theta'\nabla u))+
\int^{+\infty}_{-\infty}(\eta\ast(\theta\nabla
u),\nabla(\eta\ast(\theta  u_{t}))). \label{eq:b5}
\end{equation}
Using $(\ref{eq:b2})$ and the fact that $u_{t}\in
L^{m}((0,T)\times\partial\Omega)$ we can take
$\phi=\eta\ast\eta\ast(\theta u_{t})$ in $(\ref{eq:b4})$. Then,
multiplying by $\theta$, integrating from $-\infty$ to $\infty$ and
using standard properties of convolution, we can evaluate the second
term in $(\ref{eq:b5})$ in the following way:
\begin{equation}\label{eq:b6}
\begin{aligned}
\int^{+\infty}_{-\infty}\negqquad(\eta\ast(\theta\nabla
u),\nabla(\eta\ast(\theta  u_{t}))) & =
 \int^{+\infty}_{-\infty}\negquad\int_{\Gamma_{1}}\eta\ast(\theta\zeta)\eta\ast(\theta u_{t}) \\
& +  \int^{+\infty}_{-\infty}\!\!\int_{\Omega}\eta\ast(\theta
g)\eta\ast(\theta u_{t})
 -  \int^{+\infty}_{-\infty}\!\!\left\|\eta\ast(\theta
u_{t})\right\|^{2}_{2}
\end{aligned}
\end{equation}
Combining $(\ref{eq:b5})$ and $(\ref{eq:b6})$, and recalling that
$\theta=\theta_{\delta}$, we obtain the first approximate energy
identity
\begin{equation}\label{eq:b7}
\begin{aligned}
0 & =  \int^{+\infty}_{-\infty}(\eta\ast(\theta_{\delta}\nabla
u),\eta\ast(\theta'_{\delta}\nabla u))
 -  \int^{+\infty}_{-\infty}\left\|\eta\ast(\theta_{\delta} u_{t})\right\|^{2}_{2} \\
& +
\int^{+\infty}_{-\infty}\int_{\Gamma_{1}}\eta\ast(\theta_{\delta}\zeta)\eta\ast(\theta_{\delta}
u_{t})
 +  \int^{+\infty}_{-\infty}\int_{\Omega}\eta\ast(\theta_{\delta} g)\eta\ast(\theta_{\delta} u_{t}) \\
& :=  I_{1}+I_{2}+I_{3}+I_{4}.
\end{aligned}
\end{equation}
Now we examine each term in $(\ref{eq:b7})$ separately as
$\delta\rightarrow 0$ and $\varepsilon$ $(\mbox{i.e. }\eta)$ is
fixed. Since $\theta_{\delta}\rightarrow\theta_{0}$ a.e. and
\begin{alignat*}{2}
&\left\|\eta\ast(\theta_{\delta}\zeta)\right\|_{m',\Gamma_{1}}\leq\left\|\zeta\right\|_{m',\Gamma_{1}},
&\qquad& \left\|\eta\ast(\theta_{\delta}
u_{t})\right\|_{m',\Gamma_{1}}\leq\left\|u_{t}\right\|_{m',\Gamma_{1}}\\
&\left\|\eta\ast(\theta_{\delta}
g)\right\|_{2}\leq\left\|g\right\|_{2} &\qquad &
\left\|\eta\ast(\theta_{\delta}
u_{t})\right\|_{2}\leq\left\|u_{t}\right\|_{2},
\end{alignat*}
using $(\ref{eq:b8})$, $(\ref{eq:b2})$ and Lebesgue Dominated
Convergence Theorem we get the convergences
\begin{equation}\label{eq:b10}
\begin{aligned}
-I_{2}\rightarrow &
\int^{+\infty}_{-\infty}\left\|\eta\ast(\theta_{0}
u_{t})\right\|^{2}_{2}\\
I_{3}\rightarrow &
\int^{+\infty}_{-\infty}\int_{\Gamma_{1}}\eta\ast(\theta_{0}\zeta)\eta\ast(\theta_{0}
u_{t})\\
I_{4}\rightarrow
&\int^{+\infty}_{-\infty}\int_{\Omega}\eta\ast(\theta_{0}
g)\eta\ast(\theta_{0} u_{t}).
\end{aligned}
\end{equation}
Next we decompose the term $I_{1}$ as
\begin{equation*}
\begin{aligned}
I_{1}  = & \int^{+\infty}_{-\infty}(\eta\ast(\theta_{0}\nabla
u),\eta\ast(\theta_{\delta}'\nabla u))+
\int^{+\infty}_{-\infty}(\eta\ast
\left[(\theta_{\delta}-\theta_{0})\nabla u\right],\eta\ast(\theta_{\delta}'\nabla u)) \\
 := & I_{5}+I_{6}
\end{aligned}
\end{equation*}
Since $\theta_{\delta}\rightarrow\theta_{0}$ in $L^{1}(\mathbb{R})$,
by $(\ref{eq:b1})$ we have that
$\eta\ast\left[(\theta_{\delta}-\theta_{0})\nabla
u\right]\rightarrow 0$ strongly in $L^{\infty}(0,T;H)$. Moreover, by
$(\ref{eq:b1})$,
\begin{eqnarray*}
\left\|\eta\ast(\theta'_{\delta}\nabla u)\right\|_{L^{1}(0,T;H)} & \leq & \left\|\theta'_{\delta}\right\|_{L^{1}(\mathbb{R})}\left\|\eta\right\|_{L^{\infty}(\mathbb{R})}\left\|\nabla u\right\|_{L^{\infty}(0,T;H)}\nonumber\\
& \leq & 2\left\|\eta\right\|_{L^{\infty}(\mathbb{R})}\left\|\nabla
u\right\|_{L^{\infty}(0,T;H)},
\end{eqnarray*}
so that $I_{6}\rightarrow 0\mbox{   as }\delta\rightarrow 0$. Next
we note that, by the properties of convolution and the specific form
of $\theta_{\delta}$,
\begin{eqnarray*}
I_{5} & = & \int^{+\infty}_{-\infty}\theta'_{\delta}(\eta\ast\eta\ast(\theta_{0}\nabla u),\nabla u)\nonumber\\
& = &
\frac{1}{\delta}\int^{s+\delta}_{s}(\eta\ast\eta\ast(\theta_{0}\nabla
u),\nabla u)
-\frac{1}{\delta}\int^{t}_{t-\delta}(\eta\ast\eta\ast(\theta_{0}\nabla
u),\nabla u).
\end{eqnarray*}
By $(\ref{eq:b9})$, the function $(\eta\ast\eta\ast(\theta_{0}\nabla
u),\nabla u)$ is continuous, so
\begin{equation}
I_{5}\rightarrow (\eta\ast\eta\ast(\theta_{0}\nabla u)(s),\nabla
u(s)) -(\eta\ast\eta\ast(\theta_{0}\nabla u)(t),\nabla
u(t))\quad\text{as $\delta\rightarrow 0$.} \label{eq:b11}
\end{equation}

Combining the convergences \eqref{eq:b10}-\eqref{eq:b11} with
\eqref{eq:b7}, recalling that $\eta=\eta_{\varepsilon}$ and letting
$\rho_{\varepsilon}=\eta_{\varepsilon}\ast\eta_{\varepsilon}$, we
obtain the second approximate energy identity
\begin{equation}\label{eq:b12}
\begin{aligned}
\left.(\rho_{\varepsilon}\ast(\theta_{0}\nabla u),\nabla
u)\right|^{t}_{s}
 = & \int^{+\infty}_{-\infty}\int_{\Gamma_{1}}\eta_{\varepsilon}\ast(\theta_{0}\zeta)\eta_{\varepsilon}\ast(\theta_{0} u_{t}) \\
 + & \int^{+\infty}_{-\infty}\int_{\Omega}\eta_{\varepsilon}\ast(\theta_{0} g)\eta_{\varepsilon}\ast(\theta_{0} u_{t}) -
  \int^{+\infty}_{-\infty}\left\|\eta_{\varepsilon}\ast(\theta_{0}
u_{t})\right\|^{2}_{2}.
\end{aligned}
\end{equation}
Now we consider the convergence of the two sides of $(\ref{eq:b12})$
as $\varepsilon\rightarrow 0$. By standard arguments, using
$(\ref{eq:b2})$ and the fact that $u_{t}\in
L^{m}((0,T)\times\Omega)$ we get that
$\eta_{\varepsilon}\ast(\theta_{0}u_{t})\rightarrow\theta_{0}u_{t}$
strongly in $L^{m}((0,T)\times\Gamma_{1})$ and in
$L^{2}((0,T)\times\Omega)$. Hence, using $(\ref{eq:b8})$ and
remembering that $g\in L^{2}((0,T)\times\Omega)$, the right-hand
side of $(\ref{eq:b12})$ goes to
\begin{eqnarray*}
\int^{+\infty}_{-\infty}\int_{\Gamma_{1}}\theta^{2}_{0}\zeta u_{t}+\int^{+\infty}_{-\infty}\int_{\Omega}\theta^{2}_{0}g u_{t}-\int^{+\infty}_{-\infty}\left\|\theta_{0}u_{t}\right\|^{2}_{2} \\
= \int^{t}_{s}\int_{\Gamma_{1}}\zeta
u_{t}+\int^{t}_{s}\int_{\Omega}g
u_{t}-\int^{t}_{s}\left\|u_{t}\right\|^{2}_{2}.
\end{eqnarray*}
Concerning the left-hand side of $(\ref{eq:b12})$, we note that
$\mbox{supp }\rho_{\varepsilon}\subset(-2\varepsilon,2\varepsilon)$,
$0\leq\rho_{\varepsilon}=O(\varepsilon^{-1})$ and
    $\int^{+\infty}_{0}\rho_{\varepsilon}=\int^{0}_{-\infty}\rho_{\varepsilon}=\frac{1}{2}\int^{+\infty}_{-\infty}\rho_{\varepsilon}=\frac{1}{2}$.
Therefore, for sufficiently small $\varepsilon$,
    \[(\rho_{\varepsilon}\ast(\theta_{0}\nabla u)(t),\nabla u(t))- \tfrac 12\left\|\nabla u(t)\right\|^{2}_{2}=
    \int^{+\infty}_{0}\rho_{\varepsilon}(\tau) (\nabla u(t-\tau)-\nabla u(t),\nabla u(t))\,d\tau.
\]
Since, by $(\ref{eq:b9})$, $\tau\mapsto(\nabla u(t-\tau)\nabla
u(t),\nabla u(t))$ is continuous and vanishes when $\tau=0$, we
conclude that, as $\varepsilon\rightarrow 0$,
    $(\rho_{\varepsilon}\ast(\theta_{0}\nabla u)(t),\nabla u(t))\rightarrow\frac{1}{2}\left\|\nabla u(t)\right\|^{2}_{2}$.
The same result, of course, continues to hold when $t$ is replaced
by $s$. Then we can pass to the limit in $(\ref{eq:b12})$ and
conclude the proof of $(\ref{eq:1.10})$. To show that
$(\ref{eq:1.9})$ holds, we note that, by $(\ref{eq:1.10})$, it
follows that $t\mapsto\left\|\nabla u(t)\right\|^{2}_{2}$ is
continuous. Now we fix $t$ in $\left[0,T\right]$ and let
$t_{k}\rightarrow t$. Using $(\ref{eq:b9})$, we have
$\left\|u(t_{k})-u(t)\right\|^{2}_{V}=\left\|u(t_{k})\right\|^{2}_{V}+\left\|u(t)\right\|^{2}_{V}-2(u(t_{k}),u(t))_{V}\rightarrow
0$ as $k\rightarrow\infty$, concluding the proof.

\def\cprime{$'$}

\end{document}